\documentclass[a4paper,11pt, reqno]{amsart}
\usepackage{geometry}
\usepackage{amsmath}
\usepackage{amssymb}
\usepackage{amsthm}
\usepackage{color}
\usepackage{adjustbox}
\usepackage{graphicx}
\usepackage[ruled,vlined]{algorithm2e}
\usepackage{subcaption}
\usepackage{booktabs}
\usepackage{float}
\usepackage{hyperref}
\usepackage{multirow}
\usepackage{multicol}
\usepackage{cancel}
\usepackage{natbib}
\usepackage{appendix}

\usepackage{setspace}
\newtheorem{thm}{Theorem}

\newtheorem{prop}{Proposition}

\newtheorem{rmk}{Remark}

\def\A{\mathcal{A}}
\def\X{\mathcal{X}}

\newcommand{\dsum}{\displaystyle\sum}
\newcommand{\dmin}{\displaystyle\min}
\newcommand{\dmax}{\displaystyle\max}

\def\R{\mathbb{R}}

\usepackage{tikz}

\usepackage{pgfplots}
\pgfplotsset{compat=newest}
\usetikzlibrary{decorations.markings}
\usepackage[normalem]{ulem}

\appendixtitleon
\appendixtitletocon

\allowdisplaybreaks

\let\origmaketitle\maketitle
\def\maketitle{
  \begingroup
  \def\uppercasenonmath##1{} 
  \let\MakeUppercase\relax 
  \origmaketitle
  \endgroup
}

\title[]{\Large A Branch-and-Price approach for the Continuous Multifacility Monotone Ordered Median Problem}

\author[V. Blanco, R. G\'azquez, D. Ponce \MakeLowercase{and} J. Puerto]{{\large V\'ictor Blanco$^\dagger$, Ricardo G\'azquez$^\dagger$, Diego Ponce$^\ddagger$ and  Justo Puerto$^\ddagger$}\medskip\\
$^\dagger$Institute of Mathematics  (IMAG) and Dpt Quant. Methods Economics \& Business, Universidad de Granada,\\
\texttt{\{vblanco,rgazquez\}@ugr.es}\\
$^\ddagger$Institute of Mathematics (IMUS) and Dpt. Stats \& OR, Universidad de Sevilla\\
\texttt{\{dponce,puerto\}@us.es}\\
}

\date{\today}

\begin{document}

\maketitle

\begin{abstract}
 In this paper, we address the Continuous Multifacility Monotone Ordered Median Problem. This problem minimizes a monotone ordered weighted median function of the distances between given demand points in $\mathbb{R}^d$ and its closest facility among the $p$ selected, also in a continuous space. We propose a new branch-and-price procedure for this problem, and two mathehuristics. One of them is a decomposition-based procedure and the other an aggregation-based heuristic. We give detailed discussions of the validity of the exact formulations and also specify the implementation details of all the solution procedures. Besides, we assess their performance in an extensive computational experience that shows the superiority of the branch-and-price approach over the compact formulation in medium-sized instances. To handle larger instances it is advisable to resort to the matheuristics that also report rather good results.
 \keywords{Continuous location, Ordered median problems, Mixed Integer Non Linear Programming, Branch-and-price, Matheuristics, Clustering.}
 \end{abstract}

\section{Introduction\label{sec:intro}}

In the last years a lot of attention has been paid to the discrete aspects of location theory and a large body of literature has been published on this topic \citep[see, e.g.,][]{Beasley85,ELP2004,EMPR2009,GLM2010,MNPV09,MNV2010,PRR2013,PT2005}. One of the reasons of this flourish is the recent development of integer programming and the success of MIP solvers.  In spite of that, the reader might notice that  the mathematical origins of this theory emerged very close to some classical continuous problems such as  the well-known Fermat-Torricelli  or Weber problem and the Simpson problem \citep[see, e.g.,][and the references therein]{DH2002,LNS2015, Nickel2005}. However, the continuous counterparts of location problems have been mostly analyzed and solved using geometric constructions valid on the plane and the three dimension space that are difficult to extend when the dimensions grow or the problems are slightly modified to include some side constraints \citep[]{BG21, BPP2017,CCMP1995,CMP1998,FMB2005,NPR2003,PR2011}. These problems although very interesting quickly fall within the field of global optimization and they become very hard to solve.  Even those problems that might be considered as \textit{easy}, as for instance the classical Weber problem with Euclidean norms, are most of the times solved with algorithms (as the Weiszfeld algorithm, \cite{W1937}), whose convergence is unknown \citep[][]{CT1989}. Moreover, most problems studied in continuous location assume that a single facility is to be located, since their multifacility counterparts lead to difficult non-convex problems \citep[]{B19,BEP2014,B1995,CMP1998, MPR19,MTE2012, Puerto2020,R1992,VRE2013}. Apart from the applicability of these problems to find the \textit{optimal} positions of telecommunication services, these problems allow to extend most of the classical clustering algorithms, as $k$-Mean or $k$-Median approaches.

Motivated by the recent advances on discrete multifacility location problems with ordered median objectives \citep[][]{Deleplanque2020,EPR21,FPP2014,LPP17,MPP2020} and the available results on conic optimization \citep{BEP2014,Puerto2020}, we want to address a family of difficult continuous multifacility location problems with ordered median objectives and distances induced by general $\ell_\tau$-norms, $\tau \geq 1$. These problems gather the essential elements of both areas (discrete and continuous) of location analysis making their solution a challenging question. 

The continuous multifacility Weber problem has been already studied using branch-and-price methods  \citep[][]{K1997,duMerle1999,RZ2007,VM2015}. In addition, in discrete location, these techniques has also been applied to the $p$-median problem \citep[see, e.g.,][]{ASV2007}. However, to the best of our knowledge, a branch-and-price approach for location problems with ordered median objectives has been only developed for the discrete version in \cite{Deleplanque2020} beyond a multisource hyperplanes application \citep[][]{BJPP21}.

Our goal in this paper is to analyze the \textit{continuous multifacility  monotone ordered median problem} (MFMOMP, for short), in which we are given a finite set of demand points, $\A$,  and the goal is to find the optimal location of  $p$ new facilities such that: (1) each demand point is allocated to a single facility; and  (2) the measure of the goodness of the solution is an ordered weighted aggregation of the distances of the demand points to their closest facility \citep[see, e.g.,][]{Nickel2005}. We consider a general framework for the problem, in which the demand points (and the new facilities) lie in $\R^d$, the distances between points and facilities are $\ell_\tau$-norms based distances for $\tau \geq 1$ and the ordered median functions are assumed to be defined by non-decreasing monotone weights. These problems have been analyzed in \cite{BEP2016} in which the authors provide a Mixed Integer Second Order Cone Optimization (MISOCO) reformulation of the problem able to solve, for the first time, problems of small to medium size (up to 50 demand points), using off-the-shell solvers.

Our contribution in this paper is to introduce a new set partitioning-like (with side constraints) reformulation for this family of problems that allows us to develop a branch-and-price algorithm for solving it. This approach gives rise to a decomposition of the original problem into a master problem (set partitioning with side constraints) and a pricing problem that consists of a special form of the maximal weighted independent set problem combined with a single facility location problem. We compare this new strategy with the one obtained by solving the MISOCO formulations using standard solvers. Our results show that it is worth to use the new reformulation since it allows us to solve larger instances and reduce the gap when the time limit is reached. Moreover, we also exploit the structure of the branch-and-price approach to develop some new matheuristics for the problem that provide good quality feasible solutions for fairly large instances of several hundreds of demand points.

The paper is organized in six sections and one appendix. Section \ref{sec:COMP} formally describes the problem considered in this paper, namely the MFMOMP, and develops MISOCO formulations. Section \ref{c4-ss31} is devoted to present the new set partitioning-like formulation and all the details of the branch-and-price algorithm proposed to solve it. There, we present how to obtain initial variables for the restricted master problem, we discuss and formulate the pricing problem and set properties for handling it and describe the branching strategies and variable selection rules implemented in our algorithm. The next section, namely Section \ref{sec:heur}, deals with some heuristics algorithms proposed to provide solutions for large-sized instances. In this section, we also describe how to solve heuristically the pricing problem which gives rise to a matheuristic algorithm consisting of applying the branch-and-price algorithm but solving the pricing problem only heuristically (without certifying optimality). Obviously, since in this case the pricing problem does not certify optimality we cannot ensure optimality for the solution of the master problem, although we always obtain feasible solutions. In addition, we also present another aggregation heuristic based on clustering strategies that provides  bounds on the errors of the obtained solutions. Section \ref{sec:computational} reports the results of an exhaustive computational study with different sets of points. There, we compare the standard formulations with the branch-and-price approach and also with the heuristic algorithms. The paper ends with some conclusions in Section \ref{sec:conclusions}. Finally, Appendix \ref{ap:norms} reports the details of the computational experiment for different norms showing the usefulness of our approach.

\section{The continuous multifacility monotone ordered median problem}\label{sec:COMP}

In this section we describe the problem under study and fix the notation for the rest of the paper.

We are given a set of $n$ demand points in $\R^d$, $\A = \{a_1, \ldots, a_n\} \subset \R^{d}$, and $p\in \mathbb{N}$ ($p>0$). Our goal is to find $p$ new facilities located in $\R^d$ that minimize a function of the closest distances from the demand points to the new facilities.  We denote the index sets of demand points and facilities by $I=\{1, \ldots, n\}$ and $J=\{1, \ldots, p\}$, respectively. 
Several elements are involved when finding the \textit{best} $p$ new facilities to provide service to the $n$ demand points. In what follows we describe them:
\begin{itemize}
\item \textit{Closeness Measure:} Given a demand point $a_i$, $i \in I$, and a server $x \in \R^d$, we use norm-based distances to measure the point-to-facility closeness. Thus, we consider the following measure for the distance between $a_i$ and $x$:
$$
\delta_i(x) = \|a_i-x\|,
$$
\noindent where $\|\cdot\|$ is a norm in $\R^d$. In particular, we will assume that the norm is polyhedral or an $\ell_\tau$-norm (with $\tau\geq1$), i.e., $\delta_i(x) = \left(\dsum_{l=1}^d |a_{il}-x_l|^\tau\right)^{\frac{1}{\tau}}$.

\item \textit{Allocation Rule}:  Given a set of $p$ new facilities, $\mathcal{X} = \{x_1, \ldots, x_p\} \subset \R^d$, and a demand point $a_i$, $i \in I$,  once all the distances  between $a_i$ and $x_j$ ($j\in J$) are calculated,  one has to allocate the point to a single facility. As usual in the literature, we assume that each point is allocated to its closest facility, i.e., the closeness measure between $a_i$ and $\X$ is:
$$
\delta_i(\X) = \min_{x \in \X} \delta_i(x),
$$
\noindent and the facility $x\in \X$, reaching such a minimum is the one where $a_i$ is allocated  to (in case of ties among facilities, a random assignment is performed).

\item \textit{Aggregation of Distances}:  Given the set of demand points $\A$, the distances $\{\delta_i(\X): i \in I\} = \{\delta_1,\ldots,\delta_n\}$ must be aggregated. To this end, we use the family of ordered median criteria. Given $\lambda \in \R_+^n$  the $\lambda$-ordered median  function is defined as:
\begin{equation}\label{omf}\tag{OM}
   \textsf{OM}_\lambda(\A; \X)=  \dsum_{i\in I} \lambda_i \;\delta_{(i)},
\end{equation}
\noindent where $\delta_{(1)}, \ldots, \delta_{(n)}$ are defined such that $\{\delta_1,\ldots,\delta_n\}$ for all $i\in I$ and $\delta_{(1)}\leq \cdots \leq \delta_{(n)}$. Some particular choices of $\lambda$-weights are shown in Table \ref{tab:lambdas}. Note that most of the classical continuous location problems can be cast under this \textit{ordered median} framework, e.g., the multisource Weber problem, $\lambda=(1, \ldots, 1)$, or the multisource $p
$-center problem, $\lambda=(0, \ldots, 0, 1)$.
\begin{table}[h]
\centering
\begin{adjustbox}{max width=1.0\textwidth}
\small
\begin{tabular}{clc}
 \hline
Notation&$\lambda$-vector&Name\\
 \cmidrule(lr){1-1}\cmidrule(lr){2-2}\cmidrule(lr){3-3}

W&$ (1,\dots,1)$&$p$-Median\\
C&$(0,\dots,0,1)$&$p$-Center\\
K
&$(0,\ldots,0,\overbrace{1,\dots,1}^k)$&$k$-Center\\
D
&$ (\alpha,\ldots,\alpha,1)$&Centdian\\
S
&$(\alpha,\ldots,\alpha,\overbrace{1,\dots,1}^k)$&$k$-Entdian\\[0.1cm]
 A
&$(0=\frac{0}{n-1},\frac{1}{n-1},\frac{2}{n-1},\dots,\frac{n-2}{n-1},\frac{n-1}{n-1}=1)$&Ascendant\\[0.12cm]
 \hline
\end{tabular}
\end{adjustbox}
\caption{Examples of Ordered Median aggregation functions \label{tab:lambdas}}
\end{table}

\end{itemize}

Summarizing all the above considerations, given a set of $n$ demand points in $\R^d$, $\A = \{a_1, \ldots, a_n\} \subset \R^{d}$ and $\lambda \in \R_+^n$ (with $0 \leq \lambda_1\leq \cdots \leq \lambda_n$) the continuous multifacility monotone ordered median problem ($\mathbf{MFMOMP}_\lambda$) is the following:
\begin{equation*} \label{mf:0}\tag{$\mathbf{MFMOMP}_\lambda$}
\dmin_{\X=\{x_1, \ldots, x_p\} \subset \R^d} \textsf{OM}_\lambda(\A;\X).
\end{equation*}

Observe that the problem above is $\mathcal{NP}$-hard since the multisource $p$-median problem is just a particular instance of  \eqref{mf:0} where $\lambda=(1, \ldots, 1)$ \citep[see][]{sherali1988np}. In the following result we provide a suitable Mixed Integer Second Order Cone Optimization (MISOCO)  formulation for the problem. 

\begin{thm}\label{thm:1}
Let $\|\cdot\|$ be an $\ell_{\tau}$-norm in $\R^d$, where $\tau = \frac{r}{s}$ with $r, s\in\mathbb{N}\setminus\{0\}$, $r>s$ and $\gcd(r,s)=1$ or a polyhedral norm. Then, \eqref{mf:0} can be formulated as a MISOCO problem.
\end{thm}

\begin{proof}
	First, assume that $\{\delta_i(\X): i \in I\} = \{\delta_1,\ldots,\delta_n\}$ are given. Then, sorting the elements and multiply them by the $\lambda$-weights can be equivalently written as the following assignment problem \citep[see][]{BEP2014,BEP2016}, whose dual problem (right side) allows to compute, alternatively, the value of  $\textsf{OM}_\lambda(\delta_1,\ldots,\delta_n)$:
	\begin{align*}
	\dsum_{i\in I} \lambda_k \delta_{(k)} =  \, \max \quad&\dsum_{i, k \in I} \lambda_k \delta_i \sigma_{ik} \qquad \qquad = & \min \quad & \dsum_{i \in I} u_i + \dsum_{k \in I} v_k\\
	\mbox{s.t.}\quad&\dsum_{k\in I} \sigma_{ik} = 1,\, \forall i \in I,& \mbox{s.t.}\quad & u_i + v_k \geq \lambda_k \delta_i,\; \forall k \in I, \\ 
	&\dsum_{i\in I} \sigma_{ik} = 1,\; \forall k \in I,&  &u_i, v_k \in\R ,\; \forall i, k \in I.\\
	&\sigma_{ik} \in [0,1],\; \forall i, k \in I.& &
	\end{align*}
	
	Now, we can embed the above representation of the ordered median aggregation of $\delta_1, \ldots, \delta_n$, into \eqref{mf:0}. On the other hand, we have to represent the allocation rule (closest distances). This family of constraints is given by
	$$
	\delta_{i} = \min_{j\in J} \|a_i - x_j\|,\; \forall i \in I.
	$$
	
	In order to represent it, we use the following set of decision variables:
	$$
	w_{ij} = \left\{\begin{array}{cl}
	1 & \mbox{if $a_i$ is allocated to $x_j$,}\\
	0 & \mbox{otherwise,}
	\end{array}\right. \;\; z_{ij} = \|a_i -x_j\|, \quad r_i = \min_{j\in J} \|a_i-x_j\|,\; \forall i \in I, j \in J. 
	$$
	
	Then, a \textit{Compact} formulation for \eqref{mf:0} is:	
	\begin{subequations}
		\makeatletter
		\def\@currentlabel{${\rm C}$}
		\makeatother
		\label{eq:C}
		\renewcommand{\theequation}{${\rm C}_{\arabic{equation}}$}
		\begin{align}
		\mathbf{(C)}\min \quad &\dsum_{i\in I} u_i + \dsum_{k \in I} v_k\nonumber\\
		\mbox{s.t.} \quad & u_i + v_k \geq \lambda_k r_{i},\; \forall i, k\in I,\label{ctr:sortf}\\
		&z_{ij} \geq \|a_i - x_j\|,\; \forall i\in I, j \in J, \label{ctr:normf} \\
		& r_i \geq  z_{ij} - M (1-w_{ij}),\; \forall i\in I, j \in J,\label{ctr:alloc1f} \\
		& \dsum_{j \in J} w_{ij} = 1,\; \forall i\in I,\label{ctr:alloc2f}\\
		& x_j \in \R^d,\; \forall j \in J, \nonumber\\
		& w_{ij} \in \{0,1\},\; \forall i\in I, j \in J,\nonumber\\
		& z_{ij} \geq 0, \; \forall i\in I,j \in J,\nonumber\\
		& r_i  \geq 0,\;  \forall i\in I,\nonumber
		\end{align}
	\end{subequations}
	
	\noindent where \eqref{ctr:alloc1f} allows to compute the distance between the points and its closest facility and \eqref{ctr:alloc2f} assures single allocation of points to facilities. Here $M$ is a big enough constant $M > \max_{i,k \in I} \|a_i - a_k\|$.
	
	Finally, in case $\|\cdot\|$ is the $\ell_{\frac{r}{s}}$-norm, constraint \eqref{ctr:normf}, as already proved in \cite{BEP2014}, can be rewritten as:
	\begin{align*}
	&t_{ijl} + a_{il} - x_{jl}\geq 0,\; \forall i\in I, j\in J, l=1, \ldots, d\\
	&t_{ijl} - a_{il} + x_{jl}\geq 0,\; \forall  i\in I, j\in J,  l=1, \ldots, d,\\
	&t_{ijl}^{r} \leq \xi^{s}_{ijl} z^{r-s}_{ij},\; \forall i\in I, j\in J,   l=1, \ldots, d,\\
	&\dsum_{l=1}^d \xi_{ijl} \leq z_{ij},\; \forall i\in I, j\in J, \\
	&\xi_{ijl} \geq 0,\; \forall i\in I, j\in J, l=1, \ldots, d.
	\end{align*}
	
	If $\|\cdot\|$ is a polyhedral norm, then, \eqref{ctr:normf}, is equivalent to:
	$$
	\dsum_{l=1}^g e_{gl} (a_{il} -x_{jl}) \leq z_{ij}, \quad \forall i \in I, j\in J, e \in {\rm Ext}_{\|\cdot\|^o},
	$$
	\noindent where ${\rm Ext}_{\|\cdot\|^o} = \{e^o_1, \ldots, e^o_g\}$ are the extreme points of the dual norm (polar set of the unit ball) of  $\| \cdot\|$ \citep[see e.g.,][]{Nickel2005,WW85}.

	The final compact formulation depends on the norm, but in any case, we have a MISOCO reformulation for the \eqref{mf:0}.
\end{proof}

Note that \eqref{mf:0} is an extension of the single-facility ordered median location problem \citep[see, e.g.,][]{BEP2014}, where apart from finding the location of $p$ new facilities, the allocation patterns between demand points and facilities are also determined. In the rest of the document, we will exploit the combinatorial nature of the problem by means of a set partitioning-like formulation which is based on the following observation:
\begin{prop}\label{prop:1}
Any optimal solution of \eqref{mf:0} is characterized by $p$ sets
\\$(S_1,x_1), \ldots, (S_p,x_p) \subseteq I \times \R^d$, such that:
\begin{enumerate}
\item $\bigcup_{j\in J} S_j = I$.
\item $S_j \cap S_{j^\prime} = \emptyset,\; j,j\prime\in J$.
\item For each $j \in J$, $x_j \in \arg\dmin_{x\in \{x_1, \ldots, x_p\}} \|a_i - x\|$, for all $i \in S_j$.
\item $(x_1, \ldots, x_p) \in \arg\min \dsum_{j\in J} \dsum_{i \in S_j} \lambda_{(i)} \|a_i - x_j\|$, where ${(i)} \in I$ such that $\|a_i-x_j\|$ is the $(i)$-th smallest element in $\{\|a_i-x_j\|: i \in S_j, j \in J\}$.
\end{enumerate}
\end{prop}

\section{A set partitioning-like formulation\label{c4-ss31}}

The compact formulation shown in the previous section is affected by the size of $p$ and $d$, and it exhibits the same limitations as many other compact formulations for continuous location models even without ordering constraints. For this reason in the following we propose an alternative set partitioning-like formulation \citep[][]{duMerle1999,duMerle2002} that tries to improve the performance of solving Problem \eqref{mf:0}.

Let $S\subset I$ be a subset of demand points that are assigned to the same facility. Let $R=(S,x)$  be the pair composed by a subset $S \subset I$ and facility $x\in \R^d$. We denote by $\delta_i^R$ the contribution of demand point $i \in S$ in the subset, and let $\delta^R=(\delta_i^R)_{(i\in S)}$ be the vector of distances induced by demand points in $S$ with respect to the facility $x$. Finally, for each pair $R=(S,x)$ we define the variable
$$
y_R=\left\{ \begin{array}{ll} 1 & \mbox{if subset $S$ is selected and its associated facility is  $x$,}\\
0 & \mbox{otherwise.} \end{array} \right.
$$
We denote by $\mathcal{R} = \{(S,x): S \subset I \text{ and } x \in \mathbb{R}^d\}$.

The set partitioning-like formulation is
	\begin{subequations}
		\makeatletter
		\def\@currentlabel{${\rm MP}$}
		\makeatother
		\label{eq:MP}
		\renewcommand{\theequation}{${\rm MP}_{\arabic{equation}}$}
	\begin{align}
	\mathbf{(MP)}\min \quad &\displaystyle\sum_{i\in I}u_i+\sum_{k \in I}v_k\label{MP:of}\\
	\mbox{s.t.} \quad&\displaystyle \sum_{R=(S,x): i\in S} y_R=1,\; \forall\,i\in I,\label{MP:eq1}\\
	&\displaystyle \sum_{R\in \mathcal{R}}y_R= p,\label{MP:eq2}\\
	&\displaystyle
	u_i+v_k\ge\lambda_k\displaystyle\sum_{R=(S,x): i\in S} \delta_i^Ry_R,\; \forall\,i,k \in I,\label{MP:eq3}\\
	&y_R\in\{0,1\},\; \forall\,R\in \mathcal{R}.\label{MP:vars}\\
	& u_i, v_k \in\R ,\; \forall i, k \in I.
	\end{align}
	\end{subequations}

The objective function \eqref{MP:of}  and constraints \eqref{MP:eq3} give the correct ordered median function of the distances from the demand points to the closest facility (see Section \ref{sec:COMP}). Constraints \eqref{MP:eq1} ensure that all demand points appear in exactly one set $S$ in each feasible solution. Exactly $p$ facilities are open due to constraint \eqref{MP:eq2}. Finally, \eqref{MP:vars} define the variables as binary.

The reader might notice that this formulation has an exponential number of variables, and therefore in the following we describe the necessary elements to address its solution by means of a branch-and-price scheme.

The crucial steps in the implementation of the branch-and-price approach are the following:
\begin{enumerate}
\item {\it Initial Pool of Variables:} Generation of initial feasible solutions induced by a set of initial subsets of demand points (and their costs). 
\item {\it Pricing Problem:} In set partitioning problems, instead of solving initially the  problem with the whole set of variables, the variables have to be incorporated \textit{on-the-fly} by solving adequate pricing subproblems derived from previously computed solutions until the optimality of the solution is guaranteed. The pricing problem is derived from the relaxed version of the master problem and using the strong duality properties of the induced Linear Programming Problem.
\item {\it Branching:} The rule that creates new nodes of the branch-and-bound tree when a fractional solution is found at a node of the search tree. We have adapted the Ryan and Foster branching rule to our problem.
\item {\it Stabilization}: The convergence of column generation approaches can be sometimes erratic since the values of dual variables in the first iterations might oscillate, leading to variables of the master problem that will never appear in the optimal solution of the problem. Stabilization tries to avoid that behaviour.
\end{enumerate}

In what follows we describe how each of the items above is implemented in our proposal.

\subsection{Initial variables}

In the solution process of the set partitioning-like formulation using a branch-and-price approach, it is convenient to generate an initial pool of variables before starting solving the problem. The adequate selection of these initial variables might help to reduce the CPU time required to solve the problem. We apply an iterative strategy to generate this initial pool of $y$-variables.  In the first iteration, we randomly generate $p$ positions for the facilities. The demand points are then allocated to their closest facilities and at most $p$ subsets of demand points are generated. We incorporate the variables associated to these subsets  to the master problem \eqref{eq:MP}.  In the next iterations,  instead of generating $p$ new facilities, we keep those with more associated demand points and randomly generate the remainder.  After a fixed number of iterations, we have columns to define the restricted master problem and also an upper bound of our problem. Since the optimal position of the facilities belongs to a bounded set contained in the rectangular hull of the demand points, the random facilities are generated in the minimum hyperrectangle containing $\A$.

\subsection{The pricing problem\label{c4-sss31}}

To apply the column generation procedure we define the  restricted relaxed master of \eqref{eq:MP}, in the following \eqref{RRMP}.
	\begin{align*}\label{RRMP}\tag{RRMP}
	\min \quad&\displaystyle\sum_{i \in I}u_i+\sum_{k \in I} v_k&\textbf{Dual Multipliers}\\
	\mbox{s.t.} \quad & \displaystyle \sum_{R=(S,x): i\in S}y_R \geq 1,\; \forall\,i\in I,&\alpha_i\ge 0\\
	&\displaystyle -\sum_{R\in\mathcal{R}_0}y_R \ge-p,&\gamma\geq 0\\
	&\displaystyle u_i+v_k-\lambda_k\displaystyle\sum_{R=(S,x): i\in S}\delta_i^Ry_R \geq 0,\;  \forall\,i,k\in I, & \epsilon_{ik}\geq 0\\
	&y_R \geq 0, \; \forall\,R\in \mathcal{R}_0,&\\
	&u_i, v_k \in \R,\; \forall i, k \in I, &
	\end{align*}

\noindent where $\mathcal{R}_0\in \mathcal{R}$ represents the initial pool of columns used to initialize the set partitioning-like formulation \eqref{eq:MP}. Constraints \eqref{MP:eq1}  and \eqref{MP:eq2} are slightly modified from equations to inequalities in order to get nonnegative dual multipliers. This transformations keeps the equivalence with the original formulation since coefficients affecting the $y$-variables in constraint \eqref{MP:eq3} are nonnegative.

The value of the distances is unknown beforehand because the location of facilities can be anywhere in the continuous space. Hence, its determination requires solving continuous optimization problems. 

By strong duality, the objective value of the continuous relaxation \eqref{RRMP}, can be obtained as:
\begin{align*}\tag{Dual RRMP}
\max \quad &\displaystyle\sum_{i\in I} \alpha_i-p\gamma\\
\mbox{s.t.} \quad &\displaystyle\sum_{i \in I}\epsilon_{ik} =1,\; \forall\,k\in I,\\
&\displaystyle\sum_{k \in I}\epsilon_{ik}=1,\; \forall\,i\in I,\\
&\displaystyle\sum_{i\in S}\alpha_i-\gamma-\sum_{i\in S}\sum_{k\in I} \delta_i^R \lambda_k \epsilon_{ik} \leq 0,\; \forall\,R=(S,x)\in \mathcal{R}_0,\\
&\alpha_i,\gamma, \epsilon_{ik} \geq 0,\; \forall\,i,k\in I.
\end{align*}

Hence, for any variable $y_R$ in the master its reduced cost is
$$c_R-z_R=\displaystyle-\sum_{i\in S}\alpha^*_i+\gamma^*+\sum_{i\in S}\sum_{k\in I} \delta_i^R \lambda_k\epsilon^*_{ik},$$

\noindent where $(\alpha^*,\gamma^*,\epsilon^*)$ is the dual optimal solution of the current \eqref{RRMP}.

To certify optimality of the relaxed problem one has to check implicitly that all the reduced costs for the variables not currently included in the \eqref{RRMP} are nonnegative. Otherwise new variables must be added to the pool of columns. This can be done solving the so called pricing problem.

The pricing problem consists of finding the minimum reduced cost among the non included variables. That is,
we have to find the set $S\subset I$ and the position of the facility $x$ (its coordinates) which minimizes the reduced cost. 

For a given set of dual multipliers, $(\alpha^*, \gamma^*, \epsilon^*) \geq 0$, the problem to be solved is
\begin{align*}
\displaystyle\min_{S,x} \quad &\displaystyle-\sum_{i\in S}\alpha^*_i+\gamma^*+\sum_{i\in S}\sum_{k\in I}\delta_i^{S}\lambda_k\epsilon^*_{ik}\\
\mbox{s.t.}\quad &\delta _i^S\ge\| x-a_i\|,\; \forall i\in S.
\end{align*}

This problem can be reformulated by means of a mixed integer program. We define variables $w_i =1,\; i\in I$ if the demand point belongs to $S$, zero otherwise. We also define variables $r_i,\; i\in I$ to represent the distance from demand point $i$ to facility $x$ and zero in case $w_i=0$. Finally, $z_i,\; i\in I$ are the distances from demand point $i$ to facility $x$ in any case.
\begin{align}
\min \quad & -\dsum_{i\in I} \alpha_i^*w_{i}+\gamma^*+\sum_{i\in I} c_i r_{i} \label{pricing:of}\\
\mbox{s.t.} \quad &z_i\ge \|x-a_{i}\|,\; \forall i\in I,\label{pricing:eq1}\\
&r_{i}+M(1-w_{i})\ge z_i,\; \forall i\in I,\label{pricing:eq2}\\
&w_{i}\in\{0,1\},\; \forall i\in I,\label{pricing:vars1}\\
&z_i,r_{i}\ge 0,\; \forall i\in I,\label{pricing:vars2}
\end{align}
\noindent where $M$ can be taken as the maximum distance between two demand points and
$c_i=\sum_{k\in I} \lambda_k\epsilon^*_{ik}.$

Objective function \eqref{pricing:of} is the minimum reduced cost associated to the optimal solution of the pricing problem. Constraints \eqref{pricing:eq1} define the distances. As in Section \ref{sec:COMP} this family of constraints is defined \textit{`ad hoc'} for a given norm. Constraints \eqref{pricing:eq2} set correctly $r$-variables. Finally, constraints \eqref{pricing:vars1} and \eqref{pricing:vars2} are the domain of the variables.

As it has been shown in the proof of Theorem \ref{thm:1} the above problem can be formulated as a  MISOCO problem in case of polyhedral or $\ell_\tau$-norms.

The so called \textit{Farkas pricing} should be defined when feasibility is not ensured from the beginning. We have solved this problem with our initial pool of variables. Furthermore,  the feasibility of the master problem could be lost during the branch-and-price process when branching. In our case, we claim that Farkas pricing is not necessary because the feasibility of \eqref{eq:MP} is ensured adding an artificial variable $y_{(I,x_0)}$ to satisfy \eqref{MP:eq1} which lower bound is never set to zero and $\delta_i^R,\; i\in I$, is big enough.

When the pricing problem is solved to optimality, one can obtain a theoretic lower bound even if more variables must be added. The following remark explains how the result is applied to our particular problem.

\begin{rmk}
When the \eqref{eq:MP} embedded in a branch-and-price algorithm uses binary variables and the number of them which could take value one is bounded above, it is possible to determine a theoretic lower bound \citep[see][]{DL2005}. For our problem the number of $y$-variables assuming the value 1 in any feasible solution is $p$. Therefore, the upper bound is $p$ due to \eqref{MP:eq2}. Let $z_{RRMP}$ be the current objective function of the  \eqref{RRMP} and $\overline c_R$ the reduced cost of the variable defined by $R=(S,x)$. Hence, the lower bound is
\begin{equation}LB=z_{RRMP}+p\min_{S,x}\overline c_R.\label{eq:LB}\end{equation}

It is important to remark that it can be computed in every node of the branch-and-bound tree. This fact is particularly useful at the root node to accelerate the optimality certification or for big instances where the linear relaxation is not solved within the time limit. 

\end{rmk}

For adding a variable to the master problem it suffices to find one variable $y_R$ with negative reduced cost. This search can be performed by solving exactly the pricing problem, although that might have a high computational load. Alternatively, one could also solve heuristically the pricing problem, hoping for variables with negative reduced costs. In what follows, this approach will be called the heuristic pricer. The key observation is to check if a candidate facility is promising. 

Given the coordinates of a facility, $x$, 
we construct a set $S$ compatible with the conditions of the node of the branch-and-bound tree by
allocating demand points in $S$ to $x$ whenever the reduced cost $
c_R-z_R=\gamma^*+\sum_{i\in S}e_i
< 0$, where $e_i=-\alpha^*_i+\sum_{k\in I} \delta_i(x) \lambda_k\epsilon^*_{ik}$. In that case, the variable $y_{(S,x)}$ is candidate to be added to the pool of columns. Here, we detail how the heuristic pricer algorithm is implemented at the root node. For deeper nodes in the branch-and-bound tree we refer the reader to Section \ref{section:branching}.

For the root node, there is a very easy procedure to solve this problem, just selecting the negative ones, i.e., we define $S=\{i\in I: e_i<0\}$ and, in case $c_R-z_R<0$, the variable $y_{(S,x)}$ could be added to the problem. Additionally, the region where the facility is generated can be significantly reduced, in particular to the hyperrectangle defined by demand points with negative $e_i$.

In both exact and heuristic pricer we use multiple pricing, i.e., several columns are added to the pool at each iteration, if possible. In the exact pricer we take advantage that the solver saves different solutions besides the optimal one, so it might provide us more than one column with negative reduced cost. In the heuristic pricer we add the best variables in terms of reduced cost as long as their associated reduced costs are negative.

\subsection{Branching}\label{section:branching}
When the relaxed  \eqref{eq:MP} is solved, but the solution is not integer, the next step is to define an adequate branching rule to explore the searching tree. In this problem we can apply an adaptation of the Ryan and Foster branching rule \citep[][]{RF1981}. Given a solution with fractional $y$-variables in a node, it might occur that
\begin{equation}
0<\sum_{R: i_1,i_2\in S}y_R<1 \text{ for some }i_1,i_2\in I,\; i_1< i_2.
\label{eq:fractional}
\end{equation}

Provided that this happens, in order to find an integer solution, we create the following branches from the current node:
\begin{itemize}
\item {\bf Left branch: }$i_1$ and $i_2$ must be served by different facilities.
$$\sum_{R: i_1,i_2\in S}y_R=0.$$
\item {\bf Right branch: }$i_1$ and $i_2$ must be served by the same facility.
$$\sum_{R: i_1,i_2\in S}y_R=1.$$
\end{itemize}
\begin{rmk}
The above information is easily translated to the pricing problem adding one constraint to each one of the branches: 1) $w_{i_1}+w_{i_2}\le1$ for the left branch; and 2) $w_{i_1}=w_{i_2}$ for the right branch.
\end{rmk}

It might also happen that being some $y_R$ fractional, 
$\sum_{R: i_1,i_2\in S}y_R$ is integer for all $ i_1,i_2\in I,i_1< i_2$. The following result allows us to use this branching rule and provides a procedure to recover a feasible solution encoded in the current solution of the node.

\begin{thm} \label{th:y(s,x)} 
If $\sum_{S\ni i_1,i_2}y_{(S,x)}\in \{0,1\} $ for all $ i_1,i_2\in I$ such that $i_1< i_2$ then an integer feasible solution  of \eqref{eq:MP} be determined.
\end{thm}
\begin{proof}
Let $\X_S$ be the set of all servers which  are part of a variable $y_{(S,x)}$ belonging to the pool of columns. We define $\X_S$ for all used partitions $S$.
First, it is proven in \cite{BJNSV1998} that, under the hypothesis of the theorem, the following expression holds for any set $S$ in a partition.
$$\sum_{x\in \X_S}y_{(S,x)}\in \{0,1\}.$$

If $\sum_{x\in \X_S}y_{(S,x)}=0$, then $y_{(S,x)}=0$ for all $x\in \X_S$ because of the nonnegativity of the variables. However if
\begin{equation}\label{c4:form-c1}
\sum_{x\in \X_S}y_{(S,x)}=1
\end{equation}
$y_{(S,x)}$ could be fractional for some $x\in \X_S$.

Observe that, currently, the distance associated with client $i\in S$ in the problem is
$$\delta_i^S=\sum_{x\in \X_S}y_{(S,x)}\delta_i(x).$$

Thus, from the above we construct a new demand point $x^*$ for $S$.
\begin{equation}\label{c4:form-c2}
x_{\ell}^*=\sum_{x\in \X_S}y_{(S,x)}x_{\ell},\; \forall \ell=1,\dots,d,
\end{equation}
so that $ \delta_i(x^*)\le \delta_i^S, \; \forall i\in S$.

Indeed, by the triangular inequality and by \eqref{c4:form-c1}
$$\delta_i(x^*)=\| x^*- a_i\| =\| \sum_{x\in \X_S}y_{(S,x)}(x-a_i)\|\le  \sum_{x\in \X_S}y_{(S,x)}\|x-a_i\| =\delta_i^S,$$
 for all $i\in S$. The inequality being strict unless  $x-a_i$ for all $x\in \X_S$ are collinear.
 
 Finally, we have constructed the variable $y_{(S,x^*)}=1$ as part of a feasible integer solution of the master problem \eqref{eq:MP}. Therefore, it ensures that either the solution will be binary or the fractional solution will assume the same value that an alternative optimal which is binary. 
\end{proof}

Among all the possible choices of pairs $i_1,i_2$ verifying \eqref{eq:fractional} we propose to select the one provided by the following rule:

\begin{align}
\displaystyle \arg\max_{i_1,i_2:0<\sum_{R: i_1,i_2\in S}y_R<1}\left\{\theta\min\left\{\sum_{R: i_1,i_2\in S}y_R,1-\sum_{R: i_1,i_2\in S}y_R\right\}+(1-\theta)\frac{1}{\|a_{i_1}-a_{i_2}\|}\right\}.\label{eq:thetarule} \tag{$\theta$-rule}\end{align}

This rule uses the most fractional $y$-solution, but also pays attention to the pairs of demand points which are close to each other in the solution, assuming they will be part of the same variable with value one at the optimal solution. It has been successfully applied in a related discrete ordered median problem \citep{Deleplanque2020}. We can choose $\theta\in[0,1]$: for $\theta = 0$ the closest demand points among the pairs with fractional sum will be selected; for $\theta = 1$ the most fractional branching will be applied.

The above branching rule affects to the heuristic pricer procedure since not all $S\subset I$ are compatible with the branching conditions leading to a node. In case that we have to respect some branching decisions the pricing problem gains complexity. 

Therefore, we develop a greedy algorithm which generates heuristic variables respecting the branching decision in the current node. 

This algorithm uses the information from the branching rule to build the new variable to add.
The candidate set $S$ is built by means of a greedy algorithm similar to the one presented in \cite{Sakai2003}. First, we construct a graph of incompatibilities $G=(V,E)$, with $V$ and $E$ defined as follows: for each maximal subset of demand points  $i_1< i_2< \dots<i_{m}$ that according to the branching rule have to be assigned to the same subset; next, we include a vertex $v_{i_1}$ with weight $\omega_{i_1}=\dsum_{i\in \{i_1,\dots,i_m\}}e_i$; finally, for each $v_i,v_{i'}\in V$ such that $i$ and $i'$ cannot be assigned to the same subset at the current node, we define $\{v_i,v_{i'}\}\in E$.

The subset $S$ minimizing the reduced cost for a given $x$ can be calculated solving the Maximum Weighted Independent Set problem over $G$. We solve this problem heuristically applying a variant of the algorithm proposed in \cite{Sakai2003}.

\subsection{Convergence}\label{sec:convergence}

Due to the huge number of variables that might arise in column generation procedures it is very important checking the possible degeneracy of the algorithm. Accelerating the  convergence has been traditionally afforded by means of stabilization techniques. In recent papers \citep{BCPP2021,BJPP21} it has been shown how heuristic pricers avoid degeneracy. The ideas of stabilization and heuristic pricers have in common that both do not add in the first iterations the variable with the minimum associated reduced cost which helps in the right direction.

For the sake of readability all the computational analysis is included in Section \ref{sec:computational}. There, the reader can see how our heuristic pricer needs less variables to certify optimality than the exact pricer from medium-sized instances, therefore, accelerating the convergence.

\section{Matheuristics approaches\label{sec:heur}}

\eqref{mf:0} is an $\mathcal{NP}$-hard combinatorial optimization problem and both the compact formulation and the proposed branch-and-price approach are limited by the number of demand points $(n)$ and facilities $(p)$  to be considered. Actually, as we will see in Section \ref{sec:computational}, the two exact approaches are only capable to solve, optimally, small- and medium-sized instances. In this section we derive two different matheuristic procedures capable to handle larger instances in reasonable CPU times. The first approach is based on using the branch-and-price scheme but solving only heuristically the pricing problem. The second is an aggregation based-approach that will also allow us to derive theoretical error bounds on the approximation.

\subsection{Heuristic pricer}

The matheuristic procedure described here has been successfully applied in the literature. See, e.g., \cite{AZ2021,BCPP2021,Deleplanque2020}, and the references therein. Recall that our pricing problem is $\mathcal{NP}$-hard. In order to avoid the exact procedure for large-sized instances, where not even a single iteration could be solved exactly, we propose a matheuristic. It consists of solving each pricing problem heuristically. The inconvenient of doing that is that we do not have a theoretic lower bound during the process. Nevertheless, for instances where the time limit is reached, we are able to visit more nodes in the branch-and-bound tree which could allow us to obtain better incumbent solutions than the unfinished exact procedure.

\subsection{Aggregation schemes}\label{sec:aggregation}

The second matheuristic approach that we propose is based on applying aggregation techniques to the input data (the set of demand points). This type of approaches has been successfully applied to solve large-scale continuous location problems \citep[see][]{BG21,BJPP21,BPS2018,CS90,daskin89,irawan16}. 

Let $\A=\{a_1,\ldots,a_n\}\subset \mathbb{R}^d$ be a set of demands points. In an aggregation procedure, the set $\A$ is replaced by a multiset $\A^\prime=\{a_1^\prime, \ldots, a_n^\prime\}$, where each point $a_i$ in $\A$ is assigned to a point $a^\prime_i$ in $\A^\prime$. In order to be able to solve \eqref{mf:0} for $\A^\prime$, the cardinality of the different elements of $\A^\prime$ is assumed to be smaller than the cardinality of $\A$, and then, several $a_i$ might be assigned to the same $a_i^\prime$. 

Once the points in $\A$ are aggregated into $\A^\prime$, the procedure consists of solving \eqref{mf:0} for the demand points in $\A^\prime$. We get a set of $p$ optimal facilities for the aggregated problem, $\X^*=\{x_1^\prime, \ldots, x_p^\prime\}$, associated to its objective value $\textsf{OM}_\lambda(\A^\prime;\X^\prime)^*$. These positions can also be evaluated in the original objective function of the problem for the demand points $\A$, $\textsf{OM}_\lambda(\A;\X^\prime)^*$. The following result allows us to get upper bound of the error incurred when aggregating demand points.

\begin{thm}
Let $\X^*$ be the optimal solution of \eqref{mf:0} and $\Delta = \dmax_{i=1,\ldots,n}  \|a_i - a_i^\prime\|$. Then
\begin{equation}
|  \text{\sf{OM}}_\lambda(\A; \X^*) -  \text{\sf{OM}}_\lambda(\A^\prime, \X^\prime)| \le 2 \Delta .
\end{equation}
\end{thm}

\begin{proof}
 During the proof we assume, without loss of generality, that $\sum_{i\in I} \lambda_i=1$.
By the triangular inequality and the monotonicity and sublinearity of the ordered median function we have that ${\rm \textsf{OM}}_\lambda (\A;\X) \leq \textsf{OM}_\lambda (\A^\prime;\X^*) + \textsf{OM}_\lambda (\A^\prime;\A)$.  Since $\Delta \geq \|a_i-a_i^\prime\|$ for all $i \in I$, $|\textsf{OM}_\lambda (\A;\X^*) - \textsf{OM}_\lambda(\A^\prime;\X^*)|\le \dsum_{i\in I} \lambda_i \Delta = \Delta$.

Finally, the result is obtained applying \citep[][Theorem 5]{geoffrion1977objective}.
\end{proof}

There are different strategies to reduce the dimensionality by aggregating points. In our computational experiments we consider two differentiated approaches: the \textit{$k$-Mean Clustering} (KMEAN) and the \textit{Pick The Farthest} (PTF). In the former we replace the original points by the obtained centroids while in the latter, an initial random demand point from $\A$ is chosen and the rest are selected as the farthest demand point from the last one chosen \citep[see][for further details on this procedure]{daskin89} until a the predefined number of points is reached.

\section{Computational study} \label{sec:computational}

In order to test the performance of our branch-and-price and our matheuristic approaches, we report the results of our computational experience. We consider different sets of instances used in the location literature with size ranging from 20 to 654 demand points in the plane. In all of them, the number of facilities to be located, $p$, ranges in $\{2, 5, 10\}$ and we solve the instances for the $\lambda$-vectors in Table \ref{tab:lambdas}, $\{\mbox{W, C, K, D, S, A}\}$. We set $k = \frac{n}{2}$ for the $k$-Center and $k$-Entdian, and $\alpha = 0.9$ for the Centdian and $k$-Entdian.

For the sake of readability, we restrict the computational study of this document to $\ell_1$-norm based distances. However the reader can find extensive computational results for other norms in  Appendix \ref{ap:norms}. 

The models were coded in C and solved with SCIP v.7.0.2 \citep{scip} using as optimization solver SoPLex 5.0.2 in a Mac OS Catalina with a Core Intel Xeon W clocked at 3.2 GHz
and 96 GB of RAM memory. 

\subsection{Computational performance of the branch-and-price procedure}\label{sec:compBaP}

In this section we report the results for our branch-and-price approach based on the classical dataset provided by \cite{EWC74}. From this dataset, we randomly generate five instances with sizes $n \in \{20,30,40,45\}$ and we also consider the entire complete original instance with $n=50$. Together with the number of servers $p$ and the different ordered weighted median functions (\texttt{type}), a total of 378 instances has been considered. 

Firstly, concerning convergence (Section \ref{sec:convergence}), each line in Table \ref{tab:Heur20-40} shows the average results of 45 instances, five for each type of ordered median objective function to be minimized \{W, D, S\} and $p\in\{2, 5, 10\}$, solved to optimality. The results has been split by size ($n$) and by \texttt{Heurvar}: \texttt{FALSE} when only the exact pricer is used; \texttt{TRUE} if the heuristic pricer is used and the exact pricer is called when it does not provide new columns to add. The reader can see how the necessary number of variables to certify optimality (\texttt{Vars}) is less when heuristic is applied as the size of the instance increases. Additionally, a second effect is a reduction of the CPU time (\texttt{Time}) decreasing the number of calls to the exact pricer (\texttt{Exact}) even though the number of total iterations (\texttt{Total}) is bigger. Hence, we will use the heuristic pricer for the rest of the experiments. 

 \begin{table}
  \centering{
  \begin{adjustbox}{max width=1.0\textwidth}
  \begin{tabular}{ccrrrrrrrrrr}\hline
  
    $n$& \texttt{Heurvar}& \multicolumn{2}{c}{ \texttt{Iterations}} & \texttt{Vars} & \texttt{Time} \\
     \cmidrule(lr){1-2}\cmidrule(lr){3-4}\cmidrule(lr){5-5}\cmidrule(lr){6-6}
     &  & \texttt{Exact} &\texttt{Total} \\

\cline{1-6}\multirow{2}{*}{20}	&	\texttt{FALSE}	&	13	&	13	&	2189	&	64.92	\\
	&	\texttt{TRUE}	&	4	&	23	&	2219	&	18.02	\\
\cline{1-6}\multirow{2}{*}{30}	&	\texttt{FALSE}	&	15	&	15	&	2827	&	1034.97	\\
	&	\texttt{TRUE}	&	3	&	60	&	2856	&	191.84	\\
\cline{1-6}\multirow{2}{*}{40}	&	\texttt{FALSE}	&	50	&	50	&	4713	&	9086.33	\\
	&	\texttt{TRUE}	&	13	&	136	&	4511	&	2229.21	\\
\hline
\end{tabular}%
\end{adjustbox}
   \caption{Average number of pricer iterations, variables and time using the combined heuristic and exact pricers or only using the exact pricer \label{tab:Heur20-40}}} 
\end{table}%

Secondly, we have tuned the values of $\theta$ for the branching rule \eqref{eq:thetarule} for each of the objective functions (different values for the $\lambda$-vector) based in our computational experience. In Table \ref{tab:gaptheta} we show the average gap at termination of the above-mentioned 378 instances when we apply our branch-and-price approach fixing a time limit of 2 hours.

\begin{table}[H]
\centering
  \begin{adjustbox}{max width=1.0\textwidth}
\begin{tabular}{crrrrrrr}
\hline
 \texttt{type} 	 &$\theta = 0.0$ &$\theta = 0.1$ &$\theta = 0.3$ &$\theta = 0.5$ &$\theta = 0.7$ &$\theta = 0.9$ &$\theta = 1.0$\\
\cmidrule(lr){1-1}\cmidrule(lr){2-2}\cmidrule(lr){3-3}\cmidrule(lr){4-4}\cmidrule(lr){5-5}\cmidrule(lr){6-6}\cmidrule(lr){7-7}\cmidrule(lr){8-8}
 W	&	0.04	&	0.04	&	0.04	&	0.04	&	0.04	&	0.04	&\bf	0.02	\\
 C	&\bf	27.94	&	28.34	&	28.29	&	28.47	&	28.64	&	28.74	&	28.19	\\
 K	&	12.83	&	12.63	&	12.80	&\bf	12.46	&	12.73	&	13.15	&	12.88	\\
 D	&	0.09	&	0.07	&	0.09	&	0.09	&	0.09	&	0.09	&\bf	0.02	\\
 S	&	0.11	&	0.14	&	0.14	&	0.14	&	0.14	&	0.13	&\bf	0.10	\\
 A	&	7.73	&	7.66	&	7.69	&	7.71	&	7.64	&	7.73	&\bf	7.33	\\
 \hline

\end{tabular}
\end{adjustbox}
\caption{\texttt{GAP(\%)} for $\ell_1$-norm, \cite{EWC74} instances\label{tab:gaptheta}}
\end{table}

Therefore, we set $\theta = 0$ when we are in the Center problem, $\theta = 0.5$ when we are in the $k$-Center and $\theta = 1$ when we are in the $p$-median, Centdian, $k$-Entdian or Ascendant. Recall that when we use $\theta = 0$ we are selecting a pure distance branching rule. In contrast, when $\theta = 1$, we select the most fractional variable. On the other hand, when $\theta= 0.5$ we use a hybrid selection between the two extremes of the \eqref{eq:thetarule}. In the following, the above fixed parameters will be used in the computational experiments for exact and matheurisitic methods.

The average results obtained for the \cite{EWC74} instances, with a CPU time limit of 2 hours, are shown in Table \ref{tab:EilonL1}. There, for each combination of $n$ (size of the instance), $p$ (number of facilities to be located) and type (ordered median objective function to be minimized), we provide the average results for $\ell_1$-norm with a comparative between the compact formulation \eqref{eq:C} (Compact) and the branch-and-price approach (B\&P). The table is organized as follows: the first column gives the CPU time in seconds needed to solve the problem (\texttt{Time}) and within parentheses the number of unsolved instances (\texttt{\#Unsolved}); the second column shows the gap at the root node; the third one gives the gap at termination, i.e., the remaning MIP gap in percentage (\texttt{GAP(\%)}) when the time limit is reached, 0.00 otherwise; in the fourth column we show the number of variables (\texttt{Vars}) needed to solve the problem; in the fifth column we show the number of nodes (\texttt{Nodes}) explored in the branch-and-bound tree; and in the last one the RAM memory (\texttt{Memory (MB)}) in Megabytes required during the execution process is reported. Within each column, we highlight in bold the best result between the two formulations, namely Compact or B\&P.

 \begin{table}
 
  \centering{
  \begin{adjustbox}{max width=0.9\textwidth}
  \vspace*{-6cm}
  \begin{tabular}{cccrcrcrrrrrrrrrr}\hline
   $n$ & \texttt{type} &$p$ & \multicolumn{4}{c}{\texttt{Time (\#Unsolved)}} &  \multicolumn{2}{c}{\texttt{GAProot(\%)}} &  \multicolumn{2}{c}{\texttt{GAP(\%)}}  & \multicolumn{2}{c}{\texttt{Vars}} &  \multicolumn{2}{c}{\texttt{Nodes}}  &  \multicolumn{2}{c}{\texttt{Memory (MB)}}   \\
     \cmidrule(lr){1-3}\cmidrule(lr){4-7}\cmidrule(lr){8-9}\cmidrule(lr){10-11}\cmidrule(lr){12-13}\cmidrule(lr){14-15}\cmidrule(lr){16-17}
   &&&  \multicolumn{2}{c}{Compact} & \multicolumn{2}{c}{B\&P}&  Compact & B\&P & Compact & B\&P &Compact & B\&P&Compact & B\&P&Compact & B\&P\\
   \hline

\multirow{18}{*}{20}	&	\multirow{3}{*}{W}	&	2	&\bf	1.59	&\bf(	0	)&	22.90	&(	0	)&	93.92	&\bf	0.00	&	0.00	&	0.00	&\bf	224	&	2131	&	9518	&\bf	1	&\bf	4	&	103	\\
	&		&	5	&	1588.99	&(	1	)&\bf	8.34	&\bf(	0	)&	100.00	&\bf	0.00	&	3.38	&\bf	0.00	&\bf	470	&	2408	&	10967305	&\bf	1	&	1278	&\bf	49	\\
	&		&	10	&	---	&(	5	)&\bf	3.95	&\bf(	0	)&	100.00	&\bf	0.46	&	43.84	&\bf	0.00	&\bf	880	&	2127	&	19785215	&\bf	2	&	3425	&\bf	28	\\
\cline{2-17}	&	\multirow{3}{*}{C}	&	2	&\bf	0.06	&\bf(	0	)&	237.96	&(	4	)&	78.92	&\bf	22.59	&\bf	0.00	&	10.78	&\bf	224	&	97635	&\bf	7	&	4652	&\bf	4	&	2239	\\
	&		&	5	&\bf	12.58	&\bf(	0	)&	---	&(	5	)&	100.00	&\bf	29.46	&\bf	0.00	&	17.16	&\bf	470	&	15251	&	40379	&\bf	18660	&\bf	12	&	464	\\
	&		&	10	&\bf	511.69	&\bf(	2	)&	1831.83	&(	4	)&	100.00	&\bf	37.64	&\bf	7.59	&	20.28	&\bf	880	&	4243	&	7928195	&\bf	21617	&	725	&\bf	160	\\
\cline{2-17}	&	\multirow{3}{*}{K}	&	2	&\bf	0.35	&\bf(	0	)&	1412.69	&(	1	)&	91.43	&\bf	7.55	&\bf	0.00	&	1.42	&\bf	224	&	37917	&\bf	630	&	670	&\bf	3	&	953	\\
	&		&	5	&\bf	243.88	&\bf(	0	)&	404.99	&(	3	)&	100.00	&\bf	15.40	&\bf	0.00	&	3.85	&\bf	470	&	9363	&	657827	&\bf	6642	&\bf	77	&	279	\\
	&		&	10	&	32.22	&(	4	)&\bf	3156.63	&\bf(	2	)&	100.00	&\bf	18.53	&	36.95	&\bf	3.26	&\bf	880	&	4071	&	12150962	&\bf	9244	&	2265	&\bf	111	\\
\cline{2-17}	&	\multirow{3}{*}{D}	&	2	&\bf	2.18	&\bf(	0	)&	30.36	&(	0	)&	93.78	&\bf	0.03	&\bf	0.00	&	0.00	&\bf	224	&	2135	&	9222	&\bf	1	&\bf	5	&	108	\\
	&		&	5	&	1535.82	&(	1	)&\bf	12.18	&\bf(	0	)&	100.00	&\bf	0.00	&	6.69	&\bf	0.00	&\bf	470	&	2401	&	8972062	&\bf	1	&	1225	&\bf	49	\\
	&		&	10	&	5030.79	&(	4	)&\bf	6.88	&\bf(	0	)&	100.00	&\bf	0.46	&	48.19	&\bf	0.00	&\bf	880	&	2127	&	15660031	&\bf	4	&	2798	&\bf	28	\\
\cline{2-17}	&	\multirow{3}{*}{S}	&	2	&\bf	2.24	&\bf(	0	)&	54.23	&(	0	)&	93.77	&\bf	0.16	&	0.00	&	0.00	&\bf	224	&	2119	&	7677	&\bf	1	&\bf	5	&	106	\\
	&		&	5	&	1238.87	&(	1	)&\bf	15.75	&\bf(	0	)&	100.00	&\bf	0.06	&	4.40	&\bf	0.00	&\bf	470	&	2401	&	8141244	&\bf	2	&	745	&\bf	50	\\
	&		&	10	&	---	&(	5	)&\bf	7.61	&\bf(	0	)&	100.00	&\bf	0.53	&	50.12	&\bf	0.00	&\bf	880	&	2126	&	16072018	&\bf	5	&	2835	&\bf	28	\\
\cline{2-17}	&	\multirow{3}{*}{A}	&	2	&\bf	0.85	&\bf(	0	)&	783.95	&(	1	)&	91.63	&\bf	4.45	&\bf	0.00	&	0.35	&\bf	224	&	16973	&	1340	&\bf	400	&\bf	4	&	738	\\
	&		&	5	&\bf	411.21	&\bf(	0	)&	2304.77	&(	0	)&	100.00	&\bf	10.18	&\bf	0.00	&	0.00	&\bf	470	&	7405	&	878975	&\bf	1697	&\bf	126	&	222	\\
	&		&	10	&	60.27	&(	4	)&\bf	883.79	&\bf(	1	)&	100.00	&\bf	17.10	&	31.87	&\bf	1.73	&\bf	880	&	3288	&	10637608	&\bf	2721	&	1723	&\bf	79	\\
\cline{1-17}\multirow{18}{*}{30}	&	\multirow{3}{*}{W}	&	2	&\bf	139.91	&\bf(	0	)&	526.92	&(	0	)&	93.86	&\bf	0.00	&	0.00	&	0.00	&\bf	334	&	3142	&	787145	&\bf	1	&\bf	38	&	260	\\
	&		&	5	&	---	&(	5	)&\bf	64.66	&\bf(	0	)&	100.00	&\bf	0.00	&	52.05	&\bf	0.00	&\bf	700	&	2963	&	17382888	&\bf	1	&	8647	&\bf	109	\\
	&		&	10	&	---	&(	5	)&\bf	19.51	&\bf(	0	)&	100.00	&\bf	0.00	&	76.26	&\bf	0.00	&\bf	1310	&	2472	&	12250097	&\bf	1	&	4692	&\bf	55	\\
\cline{2-17}	&	\multirow{3}{*}{C}	&	2	&\bf	0.11	&\bf(	0	)&	39.44	&(	4	)&	79.19	&\bf	21.41	&\bf	0.00	&	15.46	&\bf	334	&	125429	&\bf	66	&	931	&\bf	8	&	1443	\\
	&		&	5	&\bf	30.64	&\bf(	0	)&	1564.58	&(	4	)&	100.00	&\bf	31.68	&\bf	0.00	&	22.73	&\bf	700	&	30216	&	69019	&\bf	2817	&\bf	19	&	389	\\
	&		&	10	&\bf	4212.55	&\bf(	3	)&	---	&(	5	)&	100.00	&\bf	34.18	&\bf	16.67	&	27.51	&\bf	1310	&	12928	&	9619002	&\bf	6027	&	1823	&\bf	190	\\
\cline{2-17}	&	\multirow{3}{*}{K}	&	2	&\bf	4.44	&\bf(	0	)&	409.69	&(	4	)&	90.88	&\bf	8.65	&\bf	0.00	&	7.58	&\bf	334	&	45846	&	8511	&\bf	147	&\bf	10	&	1696	\\
	&		&	5	&	2956.65	&(	4	)&\bf	5199.43	&\bf(	3	)&	100.00	&\bf	12.01	&	17.79	&\bf	5.82	&\bf	700	&	18893	&	12169516	&\bf	815	&	2570	&\bf	534	\\
	&		&	10	&	---	&(	5	)&\bf	2740.67	&\bf(	4	)&	100.00	&\bf	18.84	&	69.60	&\bf	12.31	&\bf	1310	&	7416	&	9299590	&\bf	2992	&	3105	&\bf	187	\\
\cline{2-17}	&	\multirow{3}{*}{D}	&	2	&\bf	201.28	&\bf(	0	)&	454.39	&(	0	)&	93.77	&\bf	0.00	&	0.00	&	0.00	&\bf	334	&	3087	&	757445	&\bf	1	&\bf	49	&	258	\\
	&		&	5	&	---	&(	5	)&\bf	65.46	&\bf(	0	)&	100.00	&\bf	0.00	&	57.16	&\bf	0.00	&\bf	700	&	2957	&	9914066	&\bf	1	&	7439	&\bf	111	\\
	&		&	10	&	---	&(	5	)&\bf	21.25	&\bf(	0	)&	100.00	&\bf	0.00	&	79.34	&\bf	0.00	&\bf	1310	&	2464	&	10108803	&\bf	1	&	4631	&\bf	55	\\
\cline{2-17}	&	\multirow{3}{*}{S}	&	2	&\bf	203.04	&\bf(	0	)&	370.63	&(	0	)&	93.68	&\bf	0.00	&\bf	0.00	&	0.00	&\bf	334	&	3184	&	566382	&\bf	1	&\bf	41	&	263	\\
	&		&	5	&	---	&(	5	)&\bf	160.85	&\bf(	0	)&	100.00	&\bf	0.03	&	56.47	&\bf	0.00	&\bf	700	&	2963	&	9283122	&\bf	2	&	7054	&\bf	112	\\
	&		&	10	&	---	&(	5	)&\bf	42.86	&\bf(	0	)&	100.00	&\bf	0.09	&	79.91	&\bf	0.00	&\bf	1310	&	2469	&	9530286	&\bf	3	&	4686	&\bf	56	\\
\cline{2-17}	&	\multirow{3}{*}{A}	&	2	&\bf	21.89	&\bf(	0	)&	3640.13	&(	2	)&	91.15	&\bf	4.46	&\bf	0.00	&	3.26	&\bf	334	&	12721	&	26764	&\bf	41	&\bf	12	&	845	\\
	&		&	5	&	5403.72	&(	4	)&\bf	2750.01	&\bf(	3	)&	100.00	&\bf	7.76	&	28.99	&\bf	2.60	&\bf	700	&	8615	&	8044660	&\bf	188	&	2288	&\bf	357	\\
	&		&	10	&	---	&(	5	)&\bf	804.71	&\bf(	4	)&	100.00	&\bf	13.38	&	70.51	&\bf	6.64	&\bf	1310	&	5529	&	7159232	&\bf	1465	&	2364	&\bf	165	\\
\cline{1-17}\multirow{18}{*}{40}	&	\multirow{3}{*}{W}	&	2	&	4028.70	&(	4	)&\bf	1675.34	&\bf(	0	)&	93.79	&\bf	0.01	&	12.34	&\bf	0.00	&\bf	444	&	5211	&	26828725	&\bf	1	&	2515	&\bf	645	\\
	&		&	5	&	---	&(	5	)&\bf	1647.86	&\bf(	0	)&	100.00	&\bf	0.02	&	67.11	&\bf	0.00	&\bf	930	&	4028	&	12240990	&\bf	3	&	10977	&\bf	229	\\
	&		&	10	&	---	&(	5	)&\bf	348.57	&\bf(	0	)&	100.00	&\bf	0.09	&	81.57	&\bf	0.00	&\bf	1740	&	4001	&	7841923	&\bf	2	&	4267	&\bf	125	\\
\cline{2-17}	&	\multirow{3}{*}{C}	&	2	&\bf	0.25	&\bf(	0	)&	---	&(	5	)&	75.52	&\bf	30.52	&\bf	0.00	&	29.73	&\bf	444	&	136451	&\bf	237	&	259	&\bf	15	&	1541	\\
	&		&	5	&\bf	116.02	&\bf(	0	)&	---	&(	5	)&	100.00	&\bf	42.30	&\bf	0.00	&	41.65	&\bf	930	&	27041	&	195892	&\bf	158	&\bf	42	&	224	\\
	&		&	10	&\bf	3022.45	&\bf(	4	)&	---	&(	5	)&	100.00	&\bf	36.47	&\bf	31.47	&	33.88	&\bf	1740	&	12733	&	7126207	&\bf	667	&	2024	&\bf	110	\\
\cline{2-17}	&	\multirow{3}{*}{K}	&	2	&\bf	58.78	&\bf(	0	)&	---	&(	5	)&	90.67	&\bf	14.52	&\bf	0.00	&	14.52	&\bf	444	&	14164	&	93918	&\bf	11	&\bf	27	&	897	\\
	&		&	5	&	---	&(	5	)&	---	&(	5	)&	100.00	&\bf	21.45	&	56.58	&\bf	21.44	&\bf	930	&	10132	&	6803627	&\bf	28	&	5632	&\bf	360	\\
	&		&	10	&	---	&(	5	)&	---	&(	5	)&	100.00	&\bf	19.04	&	75.08	&\bf	17.71	&\bf	1740	&	8823	&	5436226	&\bf	280	&	2606	&\bf	198	\\
\cline{2-17}	&	\multirow{3}{*}{D}	&	2	&	5908.68	&(	4	)&\bf	436.48	&\bf(	1	)&	93.67	&\bf	0.02	&	15.22	&\bf	0.01	&\bf	444	&	5669	&	16542227	&\bf	2	&	3164	&\bf	709	\\
	&		&	5	&	---	&(	5	)&\bf	855.62	&\bf(	1	)&	100.00	&\bf	0.11	&	68.93	&\bf	0.08	&\bf	930	&	4094	&	7984937	&\bf	2	&	10233	&\bf	233	\\
	&		&	10	&	---	&(	5	)&\bf	331.54	&\bf(	0	)&	100.00	&\bf	0.07	&	83.85	&\bf	0.00	&\bf	1740	&	4004	&	5704188	&\bf	2	&	4413	&\bf	126	\\
\cline{2-17}	&	\multirow{3}{*}{S}	&	2	&	4977.44	&(	4	)&\bf	429.96	&\bf(	1	)&	93.60	&\bf	0.47	&	14.33	&\bf	0.47	&\bf	444	&	5195	&	12853124	&\bf	1	&	2430	&\bf	657	\\
	&		&	5	&	---	&(	5	)&\bf	2159.56	&\bf(	1	)&	100.00	&\bf	0.14	&	70.18	&\bf	0.02	&\bf	930	&	4082	&	7715457	&\bf	4	&	9805	&\bf	233	\\
	&		&	10	&	---	&(	5	)&\bf	615.35	&\bf(	0	)&	100.00	&\bf	0.17	&	84.62	&\bf	0.00	&\bf	1740	&	3999	&	5409994	&\bf	4	&	4687	&\bf	126	\\
\cline{2-17}	&	\multirow{3}{*}{A}	&	2	&\bf	533.79	&\bf(	0	)&	---	&(	5	)&	90.85	&\bf	8.30	&\bf	0.00	&	8.19	&\bf	444	&	6506	&	455652	&\bf	3	&\bf	48	&	769	\\
	&		&	5	&	---	&(	5	)&	---	&(	5	)&	100.00	&\bf	14.76	&	58.65	&\bf	14.17	&\bf	930	&	5538	&	3684557	&\bf	10	&	3409	&\bf	331	\\
	&		&	10	&	---	&(	5	)&	---	&(	5	)&	100.00	&\bf	12.35	&	74.52	&\bf	10.21	&\bf	1740	&	6285	&	4403838	&\bf	161	&	2000	&\bf	214	\\
\cline{1-17}\multirow{18}{*}{45}	&	\multirow{3}{*}{W}	&	2	&	---	&(	5	)&\bf	483.59	&\bf(	1	)&	94.05	&\bf	0.04	&	27.06	&\bf	0.02	&\bf	499	&	7219	&	24989615	&\bf	2	&	5854	&\bf	1085	\\
	&		&	5	&	---	&(	5	)&\bf	1745.55	&\bf(	2	)&	100.00	&\bf	0.32	&	71.65	&\bf	0.27	&\bf	1045	&	4855	&	11473640	&\bf	4	&	11171	&\bf	374	\\
	&		&	10	&	---	&(	5	)&\bf	635.43	&\bf(	0	)&	100.00	&\bf	0.03	&	83.54	&\bf	0.00	&\bf	1955	&	4239	&	5717627	&\bf	1	&	3767	&\bf	168	\\
\cline{2-17}	&	\multirow{3}{*}{C}	&	2	&\bf	0.46	&\bf(	0	)&	---	&(	5	)&	74.99	&\bf	39.01	&\bf	0.00	&	38.99	&\bf	499	&	109398	&	628	&\bf	104	&\bf	17	&	1364	\\
	&		&	5	&\bf	144.75	&\bf(	0	)&	---	&(	5	)&	100.00	&\bf	40.69	&\bf	0.00	&	40.62	&\bf	1045	&	20483	&	215522	&\bf	31	&\bf	44	&	176	\\
	&		&	10	&	---	&(	5	)&	---	&(	5	)&	100.00	&\bf	32.80	&	37.16	&\bf	31.54	&\bf	1955	&	14204	&	6469050	&\bf	219	&	1915	&\bf	110	\\
\cline{2-17}	&	\multirow{3}{*}{K}	&	2	&\bf	342.18	&\bf(	0	)&	---	&(	5	)&	91.23	&\bf	16.93	&\bf	0.00	&	16.93	&\bf	499	&	10490	&	497310	&\bf	4	&\bf	59	&	845	\\
	&		&	5	&	---	&(	5	)&	---	&(	5	)&	100.00	&\bf	22.98	&	64.55	&\bf	22.98	&\bf	1045	&	6631	&	5434589	&\bf	11	&	6520	&\bf	295	\\
	&		&	10	&	---	&(	5	)&	---	&(	5	)&	100.00	&\bf	16.68	&	77.74	&\bf	16.48	&\bf	1955	&	8738	&	4555667	&\bf	78	&	2681	&\bf	220	\\
\cline{2-17}	&	\multirow{3}{*}{D}	&	2	&	---	&(	5	)&\bf	364.11	&\bf(	1	)&	93.96	&\bf	0.02	&	29.64	&\bf	0.02	&\bf	499	&	6473	&	14042725	&\bf	1	&	6338	&\bf	973	\\
	&		&	5	&	---	&(	5	)&\bf	1744.42	&\bf(	2	)&	100.00	&\bf	0.17	&	73.98	&\bf	0.11	&\bf	1045	&	4731	&	7322624	&\bf	3	&	10301	&\bf	365	\\
	&		&	10	&	---	&(	5	)&\bf	667.13	&\bf(	0	)&	100.00	&\bf	0.02	&	84.73	&\bf	0.00	&\bf	1955	&	4231	&	5228591	&\bf	1	&	4875	&\bf	169	\\
\cline{2-17}	&	\multirow{3}{*}{S}	&	2	&	---	&(	5	)&\bf	623.35	&\bf(	2	)&	93.87	&\bf	0.10	&	28.98	&\bf	0.09	&\bf	499	&	7260	&	10776521	&\bf	2	&	4776	&\bf	1093	\\
	&		&	5	&	---	&(	5	)&	---	&(	5	)&	100.00	&\bf	0.72	&	76.35	&\bf	0.62	&\bf	1045	&	4899	&	7356115	&\bf	7	&	10281	&\bf	378	\\
	&		&	10	&	---	&(	5	)&\bf	1848.85	&\bf(	0	)&	100.00	&\bf	0.18	&	85.38	&\bf	0.00	&\bf	1955	&	4258	&	4378226	&\bf	4	&	4283	&\bf	168	\\
\cline{2-17}	&	\multirow{3}{*}{A}	&	2	&\bf	4681.25	&\bf(	0	)&	---	&(	5	)&	91.28	&\bf	11.43	&\bf	0.00	&	11.43	&\bf	499	&	6849	&	3057137	&\bf	2	&\bf	121	&	975	\\
	&		&	5	&	---	&(	5	)&	---	&(	5	)&	100.00	&\bf	17.39	&	65.13	&\bf	17.17	&\bf	1045	&	5476	&	2139768	&\bf	4	&	2517	&\bf	415	\\
	&		&	10	&	---	&(	5	)&	---	&(	5	)&	100.00	&\bf	10.28	&	76.58	&\bf	8.88	&\bf	1955	&	6105	&	3577556	&\bf	56	&	1976	&\bf	244	\\
\cline{1-17}\multirow{18}{*}{50}	&	\multirow{3}{*}{W}	&	2	&	---	&(	1	)&\bf	331.87	&\bf(	0	)&	94.13	&\bf	0.00	&	34.44	&\bf	0.00	&\bf	554	&	8094	&	24416531	&\bf	1	&	6585	&\bf	1464	\\
	&		&	5	&	---	&(	1	)&\bf	410.87	&\bf(	0	)&	100.00	&\bf	0.00	&	76.08	&\bf	0.00	&\bf	1160	&	5292	&	9438723	&\bf	1	&	9646	&\bf	466	\\
	&		&	10	&	---	&(	1	)&\bf	1005.02	&\bf(	0	)&	100.00	&\bf	0.00	&	84.68	&\bf	0.00	&\bf	2170	&	4914	&	5017512	&\bf	1	&	3878	&\bf	225	\\
\cline{2-17}	&	\multirow{3}{*}{C}	&	2	&\bf	0.34	&\bf(	0	)&	---	&(	1	)&	75.02	&\bf	30.33	&\bf	0.00	&	30.31	&\bf	554	&	80356	&	367	&\bf	37	&\bf	22	&	1064	\\
	&		&	5	&\bf	379.06	&\bf(	0	)&	---	&(	1	)&	100.00	&\bf	41.37	&\bf	0.00	&	41.37	&\bf	1160	&	15314	&	443313	&\bf	14	&\bf	137	&	143	\\
	&		&	10	&	---	&(	1	)&	---	&(	1	)&	100.00	&\bf	37.49	&	46.67	&\bf	36.94	&\bf	2170	&	12538	&	5837062	&\bf	213	&	2937	&\bf	98	\\
\cline{2-17}	&	\multirow{3}{*}{K}	&	2	&\bf	1135.78	&\bf(	0	)&	---	&(	1	)&	91.28	&\bf	15.46	&\bf	0.00	&	15.46	&\bf	554	&	10042	&	1334361	&\bf	3	&\bf	84	&	926	\\
	&		&	5	&	---	&(	1	)&	---	&(	1	)&	100.00	&\bf	24.86	&	68.28	&\bf	24.86	&\bf	1160	&	6541	&	4368607	&\bf	4	&	6095	&\bf	324	\\
	&		&	10	&	---	&(	1	)&	---	&(	1	)&	100.00	&\bf	23.05	&	79.98	&\bf	23.04	&\bf	2170	&	7164	&	2448072	&\bf	15	&	1851	&\bf	205	\\
\cline{2-17}	&	\multirow{3}{*}{D}	&	2	&	---	&(	1	)&\bf	328.07	&\bf(	0	)&	94.07	&\bf	0.00	&	37.30	&\bf	0.00	&\bf	554	&	8035	&	12235346	&\bf	1	&	7096	&\bf	1485	\\
	&		&	5	&	---	&(	1	)&\bf	4430.48	&\bf(	0	)&	100.00	&\bf	0.08	&	78.70	&\bf	0.00	&\bf	1160	&	5415	&	5502769	&\bf	5	&	8618	&\bf	485	\\
	&		&	10	&	---	&(	1	)&\bf	1408.05	&\bf(	0	)&	100.00	&\bf	0.00	&	86.81	&\bf	0.00	&\bf	2170	&	4914	&	3617149	&\bf	2	&	4412	&\bf	219	\\
\cline{2-17}	&	\multirow{3}{*}{S}	&	2	&	---	&(	1	)&\bf	516.57	&\bf(	0	)&	93.95	&\bf	0.00	&	37.29	&\bf	0.00	&\bf	554	&	7579	&	8797114	&\bf	1	&	4912	&\bf	1387	\\
	&		&	5	&	---	&(	1	)&	---	&(	1	)&	100.00	&\bf	0.57	&	79.68	&\bf	0.57	&\bf	1160	&	5704	&	5750451	&\bf	5	&	9004	&\bf	508	\\
	&		&	10	&	---	&(	1	)&\bf	3413.97	&\bf(	0	)&	100.00	&\bf	0.00	&	87.36	&\bf	0.00	&\bf	2170	&	4962	&	3126980	&\bf	5	&	5175	&\bf	230	\\
\cline{2-17}	&	\multirow{3}{*}{A}	&	2	&	---	&(	1	)&	---	&(	1	)&	91.49	&\bf	10.36	&	19.20	&\bf	10.36	&\bf	554	&	8056	&	3163853	&\bf	2	&\bf	1161	&	1369	\\
	&		&	5	&	---	&(	1	)&	---	&(	1	)&	100.00	&\bf	19.06	&	67.25	&\bf	18.75	&\bf	1160	&	5872	&	2177800	&\bf	4	&	3290	&\bf	542	\\
	&		&	10	&	---	&(	1	)&	---	&(	1	)&	100.00	&\bf	10.17	&	78.48	&\bf	9.36	&\bf	2170	&	5764	&	2718050	&\bf	16	&	1862	&\bf	268	\\
\cline{1-17}\multicolumn{3}{c}{\bf Total Average:} 					&	645.10	&(	229	)&\bf	772.80	&\bf(	171	)&	96.71	&\bf	10.19	&	35.81	&\bf	7.98	&\bf	897	&	13958	&	6581088	&\bf	1111	&	3014	&\bf	427	\\
\hline
\end{tabular}%
\end{adjustbox}
   \caption{Results for \cite{EWC74} instances for $\ell_1$-norm \label{tab:EilonL1}}}
\end{table}%

The branch-and-price algorithm is able to solve optimally 58 instances more than the compact formulation. However, for some instances (mainly Center and $k$-Center problems or when $p=2$) the solved instances with the compact formulation need less CPU time. Thus, the first conclusion could be that when $p$ increases decomposition techniques become more important because the number of variables is not so dependant of this parameter. The second conclusion from the results is that the branch-and-price is a very  powerful tool when the gap at the root node is close to zero which does not happen when a big percentage of the positions of the $\lambda$-vector are zeros. Concerning the memory used by the tested formulations, the compact formulation needs bigger branch-and-bound trees to deal with fractional solutions whereas that branch-and-price uses more variables. 

Since the average gap at termination for the branch-and-price algorithm is much smaller than the one obtained by the compact formulation (7.98\% against 35.81\%) we will use decomposition-based algorithms to study medium- and large-sized instances.

\subsection{Computational performance of the matheuristics}

Finally, in this section, we show the performance of our matheuristics procedures to solve larger instances. Firstly, we will test them for $n=50$ where the solutions can be compared with the theoretic bounds provided by the exact method. Secondly, we will compare them using a bigger instance of 654 demand points in the following section.

In order to obtain Table \ref{tab:Heur50}, instances generated with the 50 demand points described in \cite{EWC74} are solved with 18 different configurations of ordered weighted median functions and number of open servers. A time limit of 2 hours was fixed for this experiment. Each of these 18 problems has been solved by means of the following strategies: branch-and-price procedure (B\&P); the heuristic used to generate initial columns (InitialHeur); the decomposition-based heuristic (Matheur); and the aggregation-based approaches of Section \ref{sec:aggregation} (KMEAN-20, KMEAN-30, PTF-20, PTF-30) for $|\A^\prime|=\{20,30\}$. The reported results are the CPU time and the gap (\texttt{GAP$_{LB}$(\%)}) which is calculated with respect to the lower bound of the branch-and-price algorithm when the time limit is reached. Thereby we have a theoretic gap knowing exactly the room for improvement of our heuristics. 

The branch-and-price methods report always the best performance except in one instance. In general they present less gap and, in average, it is better not wasting the time solving the exact pricer letting the algorithm go further adding columns or branching before certifying optimality. Thus, with Matheur strategy we obtain an 11.23\% of average gap. In fact, this matheuristic finds the optimal solution (certified by the exact method) at least in six instances. Concerning the time the other heuristics obtain good quality solution in much smaller CPU times. For the Eilon dataset aggregating in more demand points gives better results although the CPU time increases. However, for some instances where the application for $n=20$ certifies optimality but  $n=30$ does not, the first option could work better in a fixed time limit.

 \begin{table}
  \centering{
  \begin{adjustbox}{max width=1.0\textwidth}
  \begin{tabular}{ccrrrrrrrrrrrrrrcc}\hline
    \texttt{type} &$p$ & \multicolumn{2}{c}{ B\&P} & \multicolumn{2}{c}{ InitialHeur}& \multicolumn{2}{c}{ Matheur}& \multicolumn{2}{c}{ KMEAN-20}& \multicolumn{2}{c}{ KMEAN-30}& \multicolumn{2}{c}{ PTF-20}& \multicolumn{2}{c}{ PTF-30} \\
     \cmidrule(lr){1-2}\cmidrule(lr){3-4}\cmidrule(lr){5-6}\cmidrule(lr){7-8}\cmidrule(lr){9-10}\cmidrule(lr){11-12}\cmidrule(lr){13-14}\cmidrule(lr){15-16}
   &  & \texttt{Time} &\texttt{GAP$_{LB}$(\%)}& \texttt{Time} &\texttt{GAP$_{LB}$(\%)}& \texttt{Time} &\texttt{GAP$_{LB}$(\%)}& \texttt{Time} &\texttt{GAP$_{LB}$(\%)}& \texttt{Time} &\texttt{GAP$_{LB}$(\%)}& \texttt{Time} &\texttt{GAP$_{LB}$(\%)}& \texttt{Time} &\texttt{GAP$_{LB}$(\%)} \\
   \hline

\cline{1-16}\multirow{3}{*}{W}	&	2	&	331.87	&\bf	0.00	&	0.00	&	1.44	&	134.54	&\bf	0.00	&	74.60	&	3.10	&	5204.85	&	2.80	&	35.14	&	7.58	&	7200.18	&	2.50	\\
	&	5	&	410.87	&\bf	0.00	&	0.00	&	11.53	&	9.51	&\bf	0.00	&	21.47	&	12.55	&	398.93	&	11.25	&	16.38	&	17.88	&	187.03	&	7.90	\\
	&	10	&	1005.02	&\bf	0.00	&	0.00	&	21.57	&	2.84	&\bf	0.00	&	4.23	&	17.16	&	59.08	&	8.79	&	16.11	&	23.45	&	57.35	&	9.88	\\
\cline{1-16}\multirow{3}{*}{C}	&	2	&	7200.64	&	43.49	&	0.00	&	43.49	&	6.61	&	39.68	&	7200.16	&	47.05	&	7200.52	&	32.36	&	7200.21	&\bf	31.46	&	7200.47	&	41.39	\\
	&	5	&	7200.19	&	70.56	&	0.00	&	73.45	&	7200.00	&\bf	46.66	&	7200.09	&	82.76	&	7200.10	&	76.52	&	7200.09	&	94.74	&	7200.09	&	72.70	\\
	&	10	&	7200.19	&	58.58	&	0.00	&	101.43	&	7200.18	&\bf	28.93	&	7200.16	&	104.98	&	7200.19	&	73.44	&	7200.16	&	154.72	&	7200.18	&	70.54	\\
\cline{1-16}\multirow{3}{*}{K}	&	2	&	7200.13	&\bf	18.28	&	0.00	&	21.40	&	227.64	&	19.79	&	3589.02	&	21.19	&	7200.40	&	22.00	&	7200.35	&	20.05	&	7200.57	&	18.81	\\
	&	5	&	7200.34	&	33.09	&	0.00	&	33.09	&	392.07	&\bf	22.37	&	7200.10	&	26.84	&	7200.15	&	36.08	&	7200.10	&	42.99	&	7200.13	&	30.43	\\
	&	10	&	7200.28	&	29.94	&	0.00	&	41.73	&	864.26	&\bf	11.79	&	7200.17	&	35.38	&	7200.20	&	19.91	&	7200.17	&	41.97	&	7200.19	&	15.93	\\
\cline{1-16}\multirow{3}{*}{D}	&	2	&	328.07	&\bf	0.00	&	0.00	&	1.45	&	130.00	&\bf	0.00	&	42.82	&	3.76	&	5213.35	&	2.82	&	18.77	&	7.70	&	5105.83	&	2.47	\\
	&	5	&	4430.48	&\bf	0.00	&	0.00	&	11.43	&	10.73	&\bf	0.00	&	7.87	&	12.53	&	162.45	&	10.95	&	17.50	&	16.36	&	191.88	&	8.13	\\
	&	10	&	1408.05	&\bf	0.00	&	0.00	&	21.38	&	3.06	&	0.43	&	4.76	&	17.00	&	50.22	&	10.18	&	22.26	&	23.27	&	135.57	&	9.07	\\
\cline{1-16}\multirow{3}{*}{S}	&	2	&	516.57	&\bf	0.00	&	0.00	&	1.62	&	111.73	&\bf	0.00	&	25.43	&	2.97	&	7200.20	&	2.59	&	35.77	&	7.34	&	7200.27	&	2.33	\\
	&	5	&	7200.43	&\bf	0.57	&	0.00	&	12.28	&	14.96	&	0.64	&	36.68	&	12.60	&	498.07	&	11.81	&	41.88	&	18.33	&	746.57	&	8.26	\\
	&	10	&	3413.97	&\bf	0.00	&	0.00	&	21.24	&	3.22	&	0.66	&	10.99	&	16.86	&	29.98	&	7.98	&	50.52	&	23.29	&	114.10	&	8.97	\\
\cline{1-16}\multirow{3}{*}{A}	&	2	&	7200.39	&	11.56	&	0.00	&	11.56	&	155.18	&	10.91	&	2611.77	&\bf	9.05	&	7200.15	&	10.62	&	7200.41	&	13.49	&	7200.09	&	13.69	\\
	&	5	&	7200.40	&	23.08	&	0.00	&	23.61	&	325.35	&\bf	14.00	&	7200.10	&	19.96	&	7200.16	&	18.63	&	7200.11	&	26.03	&	7200.13	&	22.18	\\
	&	10	&	7200.24	&	10.33	&	0.00	&	31.65	&	129.77	&\bf	6.29	&	7200.17	&	17.48	&	7200.21	&	16.63	&	7200.17	&	33.49	&	7200.21	&	21.82	\\
\cline{1-16}\multicolumn{2}{c}{\bf Total Average:} 			&	4658.23	&	16.64	&	0.00	&	26.97	&	940.09	&\bf	11.23	&	3157.25	&	25.73	&	4645.51	&	20.85	&	3614.23	&	33.56	&	4763.38	&	20.39	\\

\hline
\end{tabular}%
\end{adjustbox}
   \caption{Heuristic results for instances of $n=50$, \cite{EWC74} \label{tab:Heur50}}}
\end{table}%

Table \ref{tab:Heur654} presents a similar structure that the previous table. However for $n=654$ the branch-and-price is not able to give us a good lower bound even increasing the time limit to 24 hours or using \eqref{eq:LB}. Hence, here we calculate \texttt{GAP$_{Best}$(\%)} as the gap with respect to the best known integer solution. For these large-sized problems the best solutions are found by the decomposition-based matheuristic in average, but the improvement from the initial heuristic is null for some cases.

Some instances have the best performance using KMEAN-20 or PTF-20 matheusristics. It is not appreciated a big improvement taking 30 points instead of 20 for the aggregation method. To find an explanation for that, Figure \ref{fig:626Wp5} depicts the aggregation (triangular points) and the solution (square points) for a particular instance. The reader can see how the demand points are concentrated by zones. Adding more points to  $\A^\prime$ gives an importance to some aggregated points that does not represent properly the original data of this instance of $n=654$. In this case, we can see an example for which the aggregation algorithm works better under the {\em less is more} paradigm.

 \begin{table}
  \centering{
  \begin{adjustbox}{max width=1.0\textwidth}
  \begin{tabular}{ccrrrrrrrrrrrrrrcc}\hline
    \texttt{type} &$p$ & \multicolumn{2}{c}{ B\&P} & \multicolumn{2}{c}{ InitialHeur}& \multicolumn{2}{c}{ Matheur}& \multicolumn{2}{c}{ KMEAN-20}& \multicolumn{2}{c}{ KMEAN-30}& \multicolumn{2}{c}{ PTF-20}& \multicolumn{2}{c}{ PTF-30} \\
     \cmidrule(lr){1-2}\cmidrule(lr){3-4}\cmidrule(lr){5-6}\cmidrule(lr){7-8}\cmidrule(lr){9-10}\cmidrule(lr){11-12}\cmidrule(lr){13-14}\cmidrule(lr){15-16}
   &  & \texttt{Time} &\texttt{GAP$_{Best}$(\%)}& \texttt{Time} &\texttt{GAP$_{Best}$(\%)}& \texttt{Time} &\texttt{GAP$_{Best}$(\%)}& \texttt{Time} &\texttt{GAP$_{Best}$(\%)}& \texttt{Time} &\texttt{GAP$_{Best}$(\%)}& \texttt{Time} &\texttt{GAP$_{Best}$(\%)}& \texttt{Time} &\texttt{GAP$_{Best}$(\%)} \\
   \hline

\cline{1-16}\multirow{3}{*}{W}	&	2	&	86441.18	&	9.74	&	0.06	&	9.74	&	86441.18	&	9.74	&	13.31	& \bf	0.00	&	755.99	&	30.34	&	32.59	&	10.32	&	697.98	&	29.60	\\
	&	5	&	86444.11	& \bf	0.00	&	0.06	&	0.51	&	86444.11	& \bf	0.00	&	6.56	&	28.51	&	97.79	&	58.33	&	7.46	&	49.02	&	74.85	&	55.91	\\
	&	10	&	86439.34	& \bf	0.00	&	0.06	&	15.39	&	86439.34	& \bf	0.00	&	2.64	&	74.95	&	31.49	&	84.48	&	1.58	&	123.62	&	68.11	&	103.19	\\
\cline{1-16}\multirow{3}{*}{C}	&	2	&	86407.88	&	4.16	&	0.07	&	4.16	&	86407.32	& \bf	0.00	&	420.97	&	8.91	&	86400.83	&	3.43	&	5765.05	& \bf	0.00	&	86401.06	& \bf	0.00	\\
	&	5	&	86407.52	&	6.54	&	0.06	&	8.92	&	51003.35	&	4.47	&	86424.42	&	32.61	&	86400.09	& \bf	0.00	&	86400.15	&	20.98	&	86400.09	&	1.64	\\
	&	10	&	86408.46	&	18.78	&	0.06	&	18.78	&	83212.74	&	1.85	&	42425.62	&	29.26	&	86400.17	& \bf	0.00	&	86400.16	&	60.83	&	86400.13	&	0.64	\\
\cline{1-16}\multirow{3}{*}{K}	&	2	&	86424.57	&	0.92	&	0.06	&	0.92	&	86424.57	&	0.92	&	274.39	& \bf	0.00	&	40549.49	&	10.08	&	365.79	&	8.36	&	15942.12	&	6.04	\\
	&	5	&	86424.89	& \bf	0.00	&	0.06	& \bf	0.00	&	86424.89	& \bf	0.00	&	854.51	&	24.50	&	86400.13	&	47.40	&	2868.28	&	34.04	&	86400.26	&	28.16	\\
	&	10	&	86425.31	& \bf	0.00	&	0.06	& \bf	0.00	&	86425.31	& \bf	0.00	&	6695.50	&	88.17	&	86400.19	&	55.15	&	13189.85	&	97.91	&	86400.14	&	70.97	\\
\cline{1-16}\multirow{3}{*}{D}	&	2	&	86440.98	&	11.02	&	0.07	&	11.02	&	86440.98	&	11.02	&	29.22	& \bf	0.00	&	737.65	&	31.24	&	23.55	&	11.31	&	669.43	&	30.96	\\
	&	5	&	86440.05	& \bf	0.00	&	0.06	&	0.51	&	86440.05	& \bf	0.00	&	15.08	&	29.39	&	123.71	&	60.16	&	6.47	&	51.81	&	96.06	&	61.08	\\
	&	10	&	86439.96	& \bf	0.00	&	0.06	&	15.35	&	86439.96	& \bf	0.00	&	7.19	&	73.21	&	21.26	&	75.23	&	3.05	&	122.59	&	106.33	&	103.25	\\
\cline{1-16}\multirow{3}{*}{S}	&	2	&	86440.27	&	8.02	&	0.07	&	8.02	&	86440.27	&	8.02	&	19.02	& \bf	0.00	&	427.39	&	27.13	&	27.17	&	9.27	&	515.89	&	25.98	\\
	&	5	&	86439.85	& \bf	0.00	&	0.07	&	0.52	&	86439.85	& \bf	0.00	&	9.14	&	28.84	&	200.77	&	57.12	&	6.68	&	49.57	&	143.32	&	59.67	\\
	&	10	&	86440.73	& \bf	0.00	&	0.06	&	15.80	&	86440.73	& \bf	0.00	&	6.69	&	75.73	&	32.23	&	81.06	&	2.44	&	122.72	&	107.15	&	103.53	\\
\cline{1-16}\multirow{3}{*}{A}	&	2	&	86439.83	&	3.58	&	0.06	&	3.59	&	86439.83	&	3.58	&	264.51	& \bf	0.00	&	11439.37	&	8.96	&	416.03	&	6.13	&	6083.81	&	7.90	\\
	&	5	&	86443.12	& \bf	0.00	&	0.06	& \bf	0.00	&	86443.12	& \bf	0.00	&	256.67	&	23.94	&	40907.89	&	38.23	&	414.57	&	31.07	&	15986.23	&	29.51	\\
	&	10	&	86438.72	& \bf	0.00	&	0.06	& \bf	0.00	&	86438.72	& \bf	0.00	&	4315.73	&	39.30	&	86400.18	&	65.88	&	29377.52	&	69.98	&	86400.14	&	66.13	\\
\cline{1-16}\multicolumn{2}{c}{\bf Total Average:} 			&	86432.60	&	3.49	&	0.06	&	6.29	&	84288.13	& \bf	2.20	&	7891.18	&	30.96	&	34095.92	&	40.79	&	12517.13	&	48.86	&	31049.62	&	43.56	\\
\hline
\end{tabular}%
\end{adjustbox}
   \caption{Heuristic results for instances of $n=654$, \cite{Beasley} \label{tab:Heur654}}}
\end{table}%

\begin{figure}	
	\begin{subfigure}{0.5\textwidth}
		\centering
	\fbox{\includegraphics[scale=0.3]{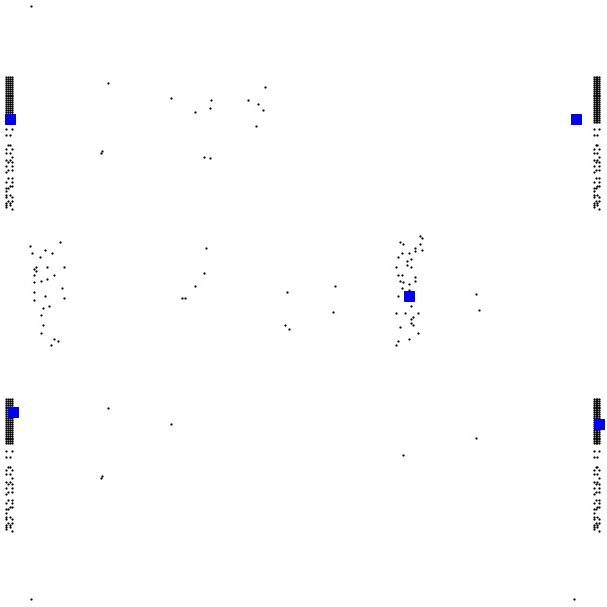}}
	\caption{InitialHeur}  \label{performanceProfilesBranching}
	\end{subfigure}%
	\begin{subfigure}{0.5\textwidth}
		\centering
	\fbox{\includegraphics[scale=0.3]{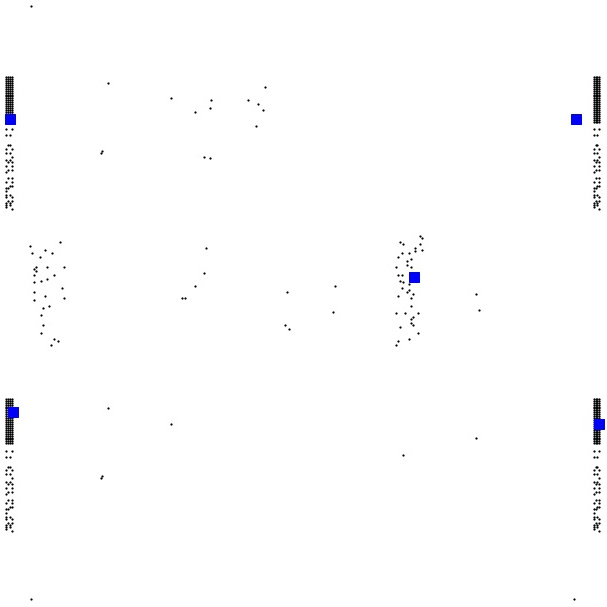}}
	\caption{Matheur}  \label{performanceProfilesBranching}
	\end{subfigure}%
		
	\begin{subfigure}{0.5\textwidth}
		\centering
	\fbox{\includegraphics[scale=0.3]{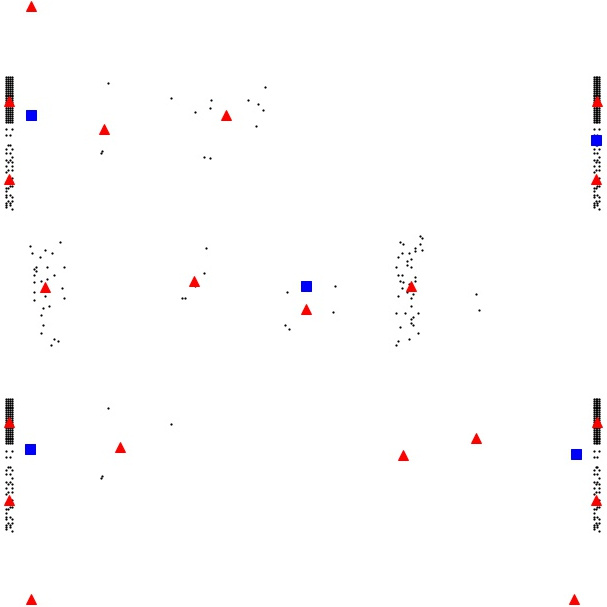}}
	\caption{KMEAN-20}  \label{performanceProfilesBranching}
	\end{subfigure}%
	\begin{subfigure}{0.5\textwidth}
		\centering
			\fbox{\includegraphics[scale=0.3]{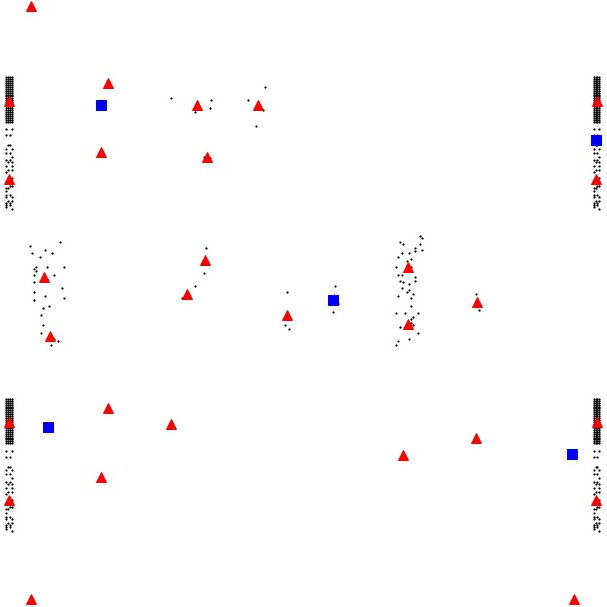}}
	\caption{KMEAN-30}  \label{performanceProfilesBranching}
	\end{subfigure}%
	\caption{Solutions for $n=654$ \citep[][]{Beasley}, W, $p=5$, and $\ell_1$-norm}  \label{fig:626Wp5}
	\end{figure}

\section{Conclusions}\label{sec:conclusions}
In this work, the Continuous Multifacility Monotone Ordered Median Problem is analyzed. This problem finds solutions in a continuous space and to solve it we have proposed two exact methods, namely a compact formulation and a branch-and-price procedure, using binary variables. Along the paper, we give full details of the branch-and-price algorithm and all its crucial steps: master problem, restricted relaxed master problem, pricing problem, initial pool of columns, feasibility, convergence, and branching. 

Moreover, theoretic and empirical results have proven the utility of the obtained lower bound. Using that bound, we have tested the new proposed matheuristics. The decomposition-based heuristics have shown a very good performance on the computational experiments. For large-sized instances, it is worth highlighting that the exact procedure has improved the initial heuristic from 6.29\% to 3.49\% in average what in real applications could make a difference. 

Among the extensive computational experiments and configurations of the problem, we highlight the usefulness of the branch-and-price approach for medium- to large-sized instances, but also the utility of the compact formulation and the aggregation-based heuristics for small values of $p$ or for some particular ordered weighted median functions.

Further research on the topic includes the design of similar branch-and-price approaches to other continuous facility location and clustering problems. Specifically, the application of set-partitioning column generation methods to hub location and covering problems with generalized upgrading \citep[see, e.g.,][]{BM19} where the index set for the $y$-variables must be adequately defined.

\section*{Acknowledgements}

The authors of this research acknowledge financial support by the Spanish Ministerio de Ciencia y Tecnolog\'ia, Agencia Estatal de Investigacion and Fondos Europeos de Desarrollo Regional (FEDER) via project PID2020-114594GB-C21. The first, third and fourth authors also acknowledge partial support from projects FEDER-US-1256951, Junta de Andalucía P18-FR-1422, CEI-3-FQM331, FQM-331, and NetmeetData: Ayudas Fundación BBVA a equipos de investigación cient\'ifica
2019. The first and second authors were partially supported by research group SEJ-584 (Junta de Andaluc\'ia). The second author was supported by Spanish Ministry of Education and Science grant number PEJ2018-002962-A and the Doctoral Program in Mathematics at the Universidad of Granada. The third author also acknowledges the grant  Contrataci\'on de Personal Investigador Doctor (Convocatoria 2019) 43 Contratos Capital Humano L\'inea 2 Paidi 2020, supported by the European Social Fund and Junta de Andaluc\'ia.

\bibliographystyle{apa}

\clearpage
\begin{appendices}

\section{Computational results for alternative $\ell_\tau$-norms}\label{ap:norms}

In this section, we show the results of our computational experiments for other $\ell_\tau$-norms, in particular, we have considered $\tau \in \{\frac{3}{2},2,3\}$. We have shown in Theorem \ref{thm:1} the general way to formulate the constraint for general values of $\tau$. In Table \ref{tab:norms}, we present the sets of constraints for the $\tau$ considered, forall $i \in I, j \in J, l \in \{1,\ldots,d\}$. Tables \ref{tab:EilonL32}, \ref{tab:EilonL2}, and \ref{tab:EilonL3} report the results.

\begin{table}[H]
	\centering
	\begin{adjustbox}{max width=1.0\textwidth}
		\begin{tabular}{lll}
			\hline
			\multicolumn{1}{c}{$\ell_{\frac{3}{2}}$	} & \multicolumn{1}{c}{$\ell_2$} & \multicolumn{1}{c}{$\ell_3$}\\
			\cmidrule(lr){1-1}\cmidrule(lr){2-2}\cmidrule(lr){3-3}
			$t_{ijl} \geq a_{il}-x_{jl},$&$t_{ijl} \geq a_{il}-x_{jl},$ & $t_{ijl} \geq a_{il}-x_{jl},$\\
			$t_{ijl} \geq -a_{il}+x_{jl},$&$t_{ijl} \geq -a_{il}+x_{jl},$ &$t_{ijl} \geq -a_{il}+x_{jl},$\\
			$z_{ij} \geq \dsum_{l=1}^d \xi_{ijl},$&$z_{ij}^2 \geq \dsum_{l=1}^d t_{ijl}^2,$ &$z_{ij} \geq \dsum_{l=1}^d \xi_{ijl}$\\
			$t^2_{ijl} \leq \psi_{ijl} \xi_{ijl},$& &$t_{ijl}^2 \leq \psi_{ijl} z_{ij},$\\
			$\psi^2_{ijl} \leq z_{ij} t_{ijl},$& &$\psi^2_{ijl} \leq  \xi_{ijl} t_{ijl},$\\
			\hline
		\end{tabular}
	\end{adjustbox}
	\caption{Constraints for different values of $\tau$ to represent $\ell_\tau$-norms\label{tab:norms}}
\end{table}

 \begin{table}
  \centering{
  \begin{adjustbox}{max width=0.9\textwidth}
  \begin{tabular}{cccrcrcrrrrrrrrrr}\hline
   $n$ & \texttt{type} &$p$ & \multicolumn{4}{c}{\texttt{Time (\#Unsolved)}} &  \multicolumn{2}{c}{\texttt{GAProot(\%)}} &  \multicolumn{2}{c}{\texttt{GAP(\%)}}  & \multicolumn{2}{c}{\texttt{Vars}} &  \multicolumn{2}{c}{\texttt{Nodes}}  &  \multicolumn{2}{c}{\texttt{Memory (MB)}}   \\
     \cmidrule(lr){1-3}\cmidrule(lr){4-7}\cmidrule(lr){8-9}\cmidrule(lr){10-11}\cmidrule(lr){12-13}\cmidrule(lr){14-15}\cmidrule(lr){16-17}
   &&&  \multicolumn{2}{c}{Compact} & \multicolumn{2}{c}{B\&P}&  Compact & B\&P & Compact & B\&P &Compact & B\&P&Compact & B\&P&Compact & B\&P\\
   \hline

\multirow{18}{*}{20}	&	\multirow{3}{*}{W}	&	2	&\bf	75.10	&\bf(	0	)&	966.81	&(	0	)&	100.00	&\bf	0.00	&	0.00	&	0.00	&	524	&\bf	500	&	31161	&\bf	1	&	187	&\bf	25	\\
	&		&	5	&	---	&(	5	)&\bf	332.73	&\bf(	0	)&	100.00	&\bf	0.00	&	58.11	&\bf	0.00	&	1270	&\bf	397	&	2323815	&\bf	1	&	11580	&\bf	12	\\
	&		&	10	&	---	&(	5	)&\bf	152.66	&\bf(	0	)&	100.00	&\bf	0.00	&	97.34	&\bf	0.00	&	2480	&\bf	377	&	1280023	&\bf	1	&	7438	&\bf	10	\\
\cline{2-17}	&	\multirow{3}{*}{C}	&	2	&\bf	0.70	&\bf(	0	)&	520.73	&(	4	)&	100.00	&\bf	21.46	&\bf	0.00	&	17.11	&\bf	524	&	13259	&\bf	317	&	650	&\bf	9	&	174	\\
	&		&	5	&\bf	57.47	&\bf(	0	)&	74.10	&(	4	)&	100.00	&\bf	34.21	&\bf	0.00	&	33.65	&\bf	1270	&	1333	&	15338	&\bf	302	&	44	&\bf	13	\\
	&		&	10	&\bf	3863.55	&\bf(	1	)&	1270.00	&(	2	)&	100.00	&\bf	18.56	&\bf	2.45	&	17.81	&	2480	&\bf	1638	&	1164203	&\bf	1936	&	1159	&\bf	15	\\
\cline{2-17}	&	\multirow{3}{*}{K}	&	2	&\bf	6.23	&\bf(	0	)&	1227.76	&(	2	)&	100.00	&\bf	6.46	&\bf	0.00	&	2.58	&\bf	524	&	7247	&	3052	&\bf	197	&\bf	20	&	212	\\
	&		&	5	&	---	&(	5	)&\bf	824.52	&\bf(	3	)&	100.00	&\bf	10.90	&	27.82	&\bf	6.21	&\bf	1270	&	2719	&	3205611	&\bf	1120	&	7563	&\bf	57	\\
	&		&	10	&	---	&(	5	)&\bf	2409.37	&\bf(	3	)&	100.00	&\bf	17.66	&	96.76	&\bf	8.78	&	2480	&\bf	1235	&	1093174	&\bf	3003	&	8496	&\bf	24	\\
\cline{2-17}	&	\multirow{3}{*}{D}	&	2	&\bf	63.65	&\bf(	0	)&	724.11	&(	0	)&	100.00	&\bf	0.00	&	0.00	&	0.00	&	524	&\bf	437	&	31224	&\bf	1	&	153	&\bf	22	\\
	&		&	5	&	---	&(	5	)&\bf	284.21	&\bf(	0	)&	100.00	&\bf	0.00	&	65.95	&\bf	0.00	&	1270	&\bf	383	&	2204675	&\bf	1	&	12441	&\bf	11	\\
	&		&	10	&	---	&(	5	)&\bf	163.89	&\bf(	0	)&	100.00	&\bf	0.08	&	95.40	&\bf	0.00	&	2480	&\bf	370	&	1140802	&\bf	3	&	9200	&\bf	9	\\
\cline{2-17}	&	\multirow{3}{*}{S}	&	2	&\bf	58.41	&\bf(	0	)&	1040.72	&(	0	)&	100.00	&\bf	0.00	&	0.00	&	0.00	&	524	&\bf	519	&	25556	&\bf	1	&	155	&\bf	25	\\
	&		&	5	&	---	&(	5	)&\bf	423.48	&\bf(	0	)&	100.00	&\bf	0.11	&	59.03	&\bf	0.00	&	1270	&\bf	397	&	2398328	&\bf	3	&	12161	&\bf	12	\\
	&		&	10	&	---	&(	5	)&\bf	175.32	&\bf(	0	)&	100.00	&\bf	0.34	&	98.74	&\bf	0.00	&	2480	&\bf	383	&	878394	&\bf	4	&	8683	&\bf	10	\\
\cline{2-17}	&	\multirow{3}{*}{A}	&	2	&\bf	16.20	&\bf(	0	)&	2802.48	&(	1	)&	100.00	&\bf	3.44	&\bf	0.00	&	0.35	&\bf	524	&	2987	&	5372	&\bf	54	&\bf	37	&	146	\\
	&		&	5	&	---	&(	5	)&\bf	2131.87	&\bf(	2	)&	100.00	&\bf	10.91	&	47.25	&\bf	3.22	&\bf	1270	&	2548	&	1996233	&\bf	413	&	8655	&\bf	69	\\
	&		&	10	&	---	&(	5	)&\bf	3440.13	&\bf(	1	)&	100.00	&\bf	11.84	&	99.32	&\bf	2.01	&	2480	&\bf	1288	&	705456	&\bf	2001	&	7031	&\bf	30	\\
\cline{1-17}\multirow{18}{*}{30}	&	\multirow{3}{*}{W}	&	2	&\bf	3007.98	&\bf(	0	)&	854.86	&(	4	)&	86.97	&\bf	19.38	&\bf	0.00	&	19.38	&\bf	784	&	852	&	1101778	&\bf	1	&	2823	&\bf	71	\\
	&		&	5	&	---	&(	5	)&\bf	3389.46	&\bf(	3	)&	87.13	&\bf	22.60	&	79.53	&\bf	22.60	&	1900	&\bf	845	&	744808	&\bf	1	&	12424	&\bf	47	\\
	&		&	10	&	---	&(	5	)&\bf	2960.90	&\bf(	0	)&	89.05	&\bf	0.00	&	88.07	&\bf	0.00	&	3710	&\bf	773	&	341608	&\bf	2	&	3070	&\bf	32	\\
\cline{2-17}	&	\multirow{3}{*}{C}	&	2	&\bf	1.27	&\bf(	0	)&	110.41	&(	4	)&	81.07	&\bf	30.05	&\bf	0.00	&	26.62	&\bf	784	&	10754	&	416	&\bf	190	&\bf	15	&	137	\\
	&		&	5	&\bf	311.13	&\bf(	0	)&	---	&(	5	)&	81.71	&\bf	45.91	&\bf	0.00	&	44.26	&\bf	1900	&	3340	&	57541	&\bf	200	&	116	&\bf	33	\\
	&		&	10	&\bf	4.95	&\bf(	4	)&	---	&(	5	)&	81.93	&\bf	45.96	&	75.89	&\bf	45.96	&	3710	&\bf	1758	&	407909	&\bf	231	&	1166	&\bf	17	\\
\cline{2-17}	&	\multirow{3}{*}{K}	&	2	&\bf	77.50	&\bf(	0	)&	19.84	&(	4	)&	85.63	&\bf	39.16	&\bf	0.00	&	38.71	&\bf	784	&	1708	&	28803	&\bf	5	&	115	&\bf	80	\\
	&		&	5	&	---	&(	5	)&\bf	20.48	&\bf(	4	)&	85.80	&\bf	18.37	&	66.83	&\bf	17.99	&\bf	1900	&	1986	&	956059	&\bf	19	&	11102	&\bf	61	\\
	&		&	10	&	---	&(	5	)&\bf	194.29	&\bf(	4	)&	86.29	&\bf	22.05	&	84.71	&\bf	21.13	&	3710	&\bf	1574	&	390091	&\bf	234	&	2707	&\bf	42	\\
\cline{2-17}	&	\multirow{3}{*}{D}	&	2	&\bf	2441.14	&\bf(	1	)&	300.79	&(	4	)&	86.55	&\bf	22.86	&\bf	2.74	&	22.86	&\bf	784	&	918	&	1208553	&\bf	1	&	3490	&\bf	77	\\
	&		&	5	&	---	&(	5	)&\bf	4831.03	&\bf(	2	)&	86.99	&\bf	24.04	&	74.25	&\bf	24.04	&	1900	&\bf	841	&	798302	&\bf	3	&	10071	&\bf	48	\\
	&		&	10	&	---	&(	5	)&\bf	2546.68	&\bf(	0	)&	87.46	&\bf	0.05	&	86.29	&\bf	0.00	&	3710	&\bf	774	&	343890	&\bf	2	&	4196	&\bf	32	\\
\cline{2-17}	&	\multirow{3}{*}{S}	&	2	&\bf	2363.80	&\bf(	0	)&	470.41	&(	4	)&	86.86	&\bf	32.35	&\bf	0.00	&	32.35	&\bf	784	&	895	&	822524	&\bf	1	&	2187	&\bf	72	\\
	&		&	5	&	---	&(	5	)&\bf	4534.78	&\bf(	2	)&	86.78	&\bf	18.98	&	76.88	&\bf	18.98	&	1900	&\bf	840	&	726885	&\bf	1	&	10244	&\bf	48	\\
	&		&	10	&	---	&(	5	)&\bf	2666.23	&\bf(	0	)&	88.10	&\bf	0.04	&	87.16	&\bf	0.00	&	3710	&\bf	776	&	410399	&\bf	5	&	4551	&\bf	31	\\
\cline{2-17}	&	\multirow{3}{*}{A}	&	2	&\bf	327.89	&\bf(	0	)&	93.44	&(	4	)&	85.31	&\bf	49.21	&\bf	0.00	&	49.18	&\bf	784	&	1259	&	70967	&\bf	2	&	365	&\bf	104	\\
	&		&	5	&	---	&(	5	)&\bf	95.22	&\bf(	4	)&	86.39	&\bf	7.72	&	70.59	&\bf	7.37	&	1900	&\bf	1346	&	626210	&\bf	9	&	8166	&\bf	74	\\
	&		&	10	&	---	&(	5	)&\bf	169.25	&\bf(	4	)&	86.14	&\bf	13.71	&	84.74	&\bf	10.99	&	3710	&\bf	1487	&	307003	&\bf	141	&	2647	&\bf	54	\\
\cline{1-17}\multirow{18}{*}{40}	&	\multirow{3}{*}{W}	&	2	&	---	&(	5	)&	---	&(	5	)&	100.00	&\bf	32.33	&	45.10	&\bf	32.33	&\bf	1044	&	1292	&	1282229	&\bf	1	&	15092	&\bf	131	\\
	&		&	5	&	---	&(	5	)&	---	&(	5	)&	100.00	&\bf	67.86	&	96.95	&\bf	67.86	&	2530	&\bf	1298	&	364182	&\bf	1	&	8105	&\bf	102	\\
	&		&	10	&	---	&(	5	)&	---	&(	5	)&	100.00	&\bf	92.83	&	100.00	&\bf	92.83	&	4940	&\bf	1339	&	165306	&\bf	1	&	2712	&\bf	83	\\
\cline{2-17}	&	\multirow{3}{*}{C}	&	2	&\bf	4.22	&\bf(	0	)&	---	&(	5	)&	100.00	&\bf	44.51	&\bf	0.00	&	44.51	&\bf	1044	&	2468	&	1043	&\bf	2	&\bf	20	&	35	\\
	&		&	5	&\bf	3879.78	&\bf(	3	)&	---	&(	5	)&	100.00	&\bf	77.56	&\bf	41.01	&	77.56	&	2530	&\bf	2525	&	606460	&\bf	2	&	4626	&\bf	27	\\
	&		&	10	&	---	&(	5	)&	---	&(	5	)&	100.00	&\bf	78.16	&	91.15	&\bf	78.16	&	4940	&\bf	2206	&	237096	&\bf	13	&	883	&\bf	22	\\
\cline{2-17}	&	\multirow{3}{*}{K}	&	2	&\bf	2889.87	&\bf(	0	)&	---	&(	5	)&	100.00	&\bf	61.39	&\bf	0.00	&	61.39	&\bf	1044	&	2226	&	614033	&\bf	1	&	2098	&\bf	125	\\
	&		&	5	&	---	&(	5	)&	---	&(	5	)&	100.00	&\bf	83.30	&	92.54	&\bf	83.30	&	2530	&\bf	2213	&	451617	&\bf	1	&	4187	&\bf	94	\\
	&		&	10	&	---	&(	5	)&	---	&(	5	)&	100.00	&\bf	75.79	&	100.00	&\bf	75.79	&	4940	&\bf	2169	&	163856	&\bf	3	&	1240	&\bf	83	\\
\cline{2-17}	&	\multirow{3}{*}{D}	&	2	&	---	&(	5	)&	---	&(	5	)&	100.00	&\bf	32.13	&	40.95	&\bf	32.13	&\bf	1044	&	1529	&	1548255	&\bf	1	&	11470	&\bf	153	\\
	&		&	5	&	---	&(	5	)&	---	&(	5	)&	100.00	&\bf	68.37	&	99.29	&\bf	68.37	&	2530	&\bf	1310	&	360994	&\bf	1	&	12324	&\bf	99	\\
	&		&	10	&	---	&(	5	)&	---	&(	5	)&	100.00	&\bf	91.94	&	100.00	&\bf	91.94	&	4940	&\bf	1357	&	141565	&\bf	1	&	4573	&\bf	79	\\
\cline{2-17}	&	\multirow{3}{*}{S}	&	2	&	---	&(	5	)&	---	&(	5	)&	100.00	&\bf	30.53	&	42.63	&\bf	30.53	&\bf	1044	&	1443	&	1109568	&\bf	1	&	14496	&\bf	145	\\
	&		&	5	&	---	&(	5	)&	---	&(	5	)&	100.00	&\bf	72.37	&	98.23	&\bf	72.37	&	2530	&\bf	1299	&	358058	&\bf	1	&	11144	&\bf	97	\\
	&		&	10	&	---	&(	5	)&	---	&(	5	)&	100.00	&\bf	90.94	&	100.00	&\bf	90.94	&	4940	&\bf	1362	&	176603	&\bf	1	&	2825	&\bf	76	\\
\cline{2-17}	&	\multirow{3}{*}{A}	&	2	&\bf	2886.85	&\bf(	4	)&	---	&(	5	)&	100.00	&\bf	61.20	&\bf	14.49	&	61.20	&\bf	1044	&	1876	&	1127650	&\bf	1	&	3383	&\bf	212	\\
	&		&	5	&	---	&(	5	)&	---	&(	5	)&	100.00	&\bf	82.57	&	93.62	&\bf	82.57	&	2530	&\bf	1656	&	402572	&\bf	1	&	4383	&\bf	134	\\
	&		&	10	&	---	&(	5	)&	---	&(	5	)&	100.00	&\bf	75.96	&	100.00	&\bf	75.96	&	4940	&\bf	1782	&	145643	&\bf	1	&	2437	&\bf	114	\\
\cline{1-17}\multirow{18}{*}{45}	&	\multirow{3}{*}{W}	&	2	&	---	&(	5	)&	---	&(	5	)&	100.00	&\bf	34.65	&	46.87	&\bf	34.65	&\bf	1174	&	1462	&	974796	&\bf	1	&	11024	&\bf	145	\\
	&		&	5	&	---	&(	5	)&	---	&(	5	)&	100.00	&\bf	83.70	&	97.65	&\bf	83.70	&	2845	&\bf	1444	&	367765	&\bf	1	&	3877	&\bf	126	\\
	&		&	10	&	---	&(	5	)&	---	&(	5	)&	100.00	&	100.00	&	100.00	&	100.00	&	5555	&\bf	1522	&	122910	&\bf	1	&	1081	&\bf	101	\\
\cline{2-17}	&	\multirow{3}{*}{C}	&	2	&\bf	6.63	&\bf(	0	)&	---	&(	5	)&	100.00	&\bf	57.24	&\bf	0.00	&	57.24	&\bf	1174	&	2438	&	1160	&\bf	1	&\bf	23	&	34	\\
	&		&	5	&\bf	4337.92	&\bf(	3	)&	---	&(	5	)&	100.00	&\bf	84.59	&\bf	49.89	&	84.59	&	2845	&\bf	2696	&	467399	&\bf	1	&	1217	&\bf	33	\\
	&		&	10	&	---	&(	5	)&	---	&(	5	)&	100.00	&\bf	64.84	&	96.12	&\bf	64.84	&	5555	&\bf	2327	&	135558	&\bf	3	&	578	&\bf	24	\\
\cline{2-17}	&	\multirow{3}{*}{K}	&	2	&\bf	6688.87	&\bf(	4	)&	---	&(	5	)&	100.00	&\bf	63.44	&\bf	14.13	&	63.44	&\bf	1174	&	2443	&	1718920	&\bf	1	&	6733	&\bf	142	\\
	&		&	5	&	---	&(	5	)&	---	&(	5	)&	100.00	&\bf	71.76	&	95.10	&\bf	71.76	&	2845	&\bf	2663	&	453635	&\bf	1	&	3364	&\bf	169	\\
	&		&	10	&	---	&(	5	)&	---	&(	5	)&	100.00	&\bf	91.29	&	99.70	&\bf	91.29	&	5555	&\bf	2231	&	158192	&\bf	1	&	781	&\bf	97	\\
\cline{2-17}	&	\multirow{3}{*}{D}	&	2	&	---	&(	5	)&	---	&(	5	)&	100.00	&\bf	30.83	&	48.64	&\bf	30.83	&\bf	1174	&	1508	&	933735	&\bf	1	&	10637	&\bf	153	\\
	&		&	5	&	---	&(	5	)&	---	&(	5	)&	100.00	&\bf	91.29	&	98.36	&\bf	91.29	&	2845	&\bf	1470	&	374986	&\bf	1	&	5117	&\bf	129	\\
	&		&	10	&	---	&(	5	)&	---	&(	5	)&	100.00	&	100.00	&	100.00	&	100.00	&	5555	&\bf	1528	&	159745	&\bf	1	&	1184	&\bf	106	\\
\cline{2-17}	&	\multirow{3}{*}{S}	&	2	&	---	&(	5	)&	---	&(	5	)&	100.00	&\bf	31.76	&	45.32	&\bf	31.76	&\bf	1174	&	1556	&	1059244	&\bf	1	&	8872	&\bf	160	\\
	&		&	5	&	---	&(	5	)&	---	&(	5	)&	100.00	&\bf	86.56	&	99.18	&\bf	86.56	&	2845	&\bf	1477	&	452467	&\bf	1	&	4738	&\bf	125	\\
	&		&	10	&	---	&(	5	)&	---	&(	5	)&	100.00	&	100.00	&	100.00	&	100.00	&	5555	&\bf	1522	&	131588	&\bf	1	&	976	&\bf	100	\\
\cline{2-17}	&	\multirow{3}{*}{A}	&	2	&	---	&(	5	)&	---	&(	5	)&	100.00	&\bf	60.42	&\bf	25.43	&	60.42	&\bf	1174	&	1992	&	861937	&\bf	1	&	4928	&\bf	221	\\
	&		&	5	&	---	&(	5	)&	---	&(	5	)&	100.00	&\bf	84.00	&	97.37	&\bf	84.00	&	2845	&\bf	1822	&	297805	&\bf	1	&	4261	&\bf	172	\\
	&		&	10	&	---	&(	5	)&	---	&(	5	)&	100.00	&\bf	93.77	&	100.00	&\bf	93.77	&	5555	&\bf	2008	&	141643	&\bf	1	&	845	&\bf	175	\\
\cline{1-17}\multirow{18}{*}{50}	&	\multirow{3}{*}{W}	&	2	&	---	&(	1	)&	---	&(	1	)&	100.00	&\bf	36.39	&	57.48	&\bf	36.39	&\bf	1304	&	1749	&	759779	&\bf	1	&	11084	&\bf	186	\\
	&		&	5	&	---	&(	1	)&	---	&(	1	)&	100.00	&\bf	88.42	&	97.49	&\bf	88.42	&	3160	&\bf	1659	&	270228	&\bf	1	&	3125	&\bf	135	\\
	&		&	10	&	---	&(	1	)&	---	&(	1	)&	100.00	&	100.00	&	100.00	&	100.00	&	6170	&\bf	1777	&	122705	&\bf	1	&	1843	&\bf	127	\\
\cline{2-17}	&	\multirow{3}{*}{C}	&	2	&\bf	6.06	&\bf(	0	)&	---	&(	1	)&	100.00	&\bf	57.05	&\bf	0.00	&	57.05	&\bf	1304	&	2813	&	879	&\bf	1	&\bf	28	&	36	\\
	&		&	5	&	---	&(	1	)&	---	&(	1	)&	100.00	&\bf	94.44	&\bf	86.11	&	94.44	&	3160	&\bf	2767	&	489141	&\bf	1	&	1274	&\bf	41	\\
	&		&	10	&	---	&(	1	)&	---	&(	1	)&	100.00	&\bf	66.91	&	100.00	&\bf	66.91	&	6170	&\bf	2678	&	117338	&\bf	2	&	535	&\bf	32	\\
\cline{2-17}	&	\multirow{3}{*}{K}	&	2	&	---	&(	1	)&	---	&(	1	)&	100.00	&\bf	66.85	&\bf	22.51	&	66.85	&\bf	1304	&	2554	&	1624297	&\bf	1	&	9722	&\bf	141	\\
	&		&	5	&	---	&(	1	)&	---	&(	1	)&	100.00	&\bf	95.48	&\bf	92.64	&	95.48	&	3160	&\bf	2458	&	501803	&\bf	1	&	1496	&\bf	132	\\
	&		&	10	&	---	&(	1	)&	---	&(	1	)&	100.00	&\bf	99.57	&	100.00	&\bf	99.57	&	6170	&\bf	2418	&	128940	&\bf	1	&	669	&\bf	115	\\
\cline{2-17}	&	\multirow{3}{*}{D}	&	2	&	---	&(	1	)&	---	&(	1	)&	100.00	&\bf	31.81	&	51.21	&\bf	31.81	&\bf	1304	&	1704	&	1059488	&\bf	1	&	10438	&\bf	182	\\
	&		&	5	&	---	&(	1	)&	---	&(	1	)&	100.00	&\bf	91.53	&	97.42	&\bf	91.53	&	3160	&\bf	1671	&	405884	&\bf	1	&	1222	&\bf	156	\\
	&		&	10	&	---	&(	1	)&	---	&(	1	)&	100.00	&	100.00	&	100.00	&	100.00	&	6170	&\bf	1710	&	156712	&\bf	1	&	1360	&\bf	127	\\
\cline{2-17}	&	\multirow{3}{*}{S}	&	2	&	---	&(	1	)&	---	&(	1	)&	100.00	&\bf	23.29	&	56.18	&\bf	23.29	&\bf	1304	&	2040	&	757576	&\bf	1	&	10234	&\bf	215	\\
	&		&	5	&	---	&(	1	)&	---	&(	1	)&	100.00	&\bf	46.71	&	98.40	&\bf	46.71	&	3160	&\bf	1772	&	404250	&\bf	1	&	2374	&\bf	166	\\
	&		&	10	&	---	&(	1	)&	---	&(	1	)&	100.00	&	100.00	&	100.00	&	100.00	&	6170	&\bf	1739	&	144453	&\bf	1	&	764	&\bf	131	\\
\cline{2-17}	&	\multirow{3}{*}{A}	&	2	&	---	&(	1	)&	---	&(	1	)&	100.00	&\bf	58.75	&\bf	40.71	&	58.75	&\bf	1304	&	2006	&	519352	&\bf	1	&	6759	&\bf	215	\\
	&		&	5	&	---	&(	1	)&	---	&(	1	)&	100.00	&\bf	86.23	&	95.15	&\bf	86.23	&	3160	&\bf	2175	&	396890	&\bf	1	&	2119	&\bf	208	\\
	&		&	10	&	---	&(	1	)&	---	&(	1	)&	100.00	&	100.00	&	100.00	&	100.00	&	6170	&\bf	2127	&	127758	&\bf	1	&	672	&\bf	183	\\
\cline{1-17}\multicolumn{3}{c}{\bf Total Average:} 					&	1016.47	&(	282	)&\bf	1438.77	&\bf(	277	)&	96.64	&\bf	44.54	&	59.19	&\bf	43.82	&	2451	&\bf	1902	&	628495	&\bf	143	&	4813	&\bf	85	\\

\hline
\end{tabular}%
\end{adjustbox}
   \caption{Results for \cite{EWC74} instances for $\ell_{\frac{3}{2}}$-norm \label{tab:EilonL32}}}
\end{table}%

 \begin{table}
  \centering{
  \begin{adjustbox}{max width=0.9\textwidth}
  \begin{tabular}{cccrcrcrrrrrrrrrr}\hline
   $n$ & \texttt{type} &$p$ & \multicolumn{4}{c}{\texttt{Time (\#Unsolved)}} &  \multicolumn{2}{c}{\texttt{GAProot(\%)}} &  \multicolumn{2}{c}{\texttt{GAP(\%)}}  & \multicolumn{2}{c}{\texttt{Vars}} &  \multicolumn{2}{c}{\texttt{Nodes}}  &  \multicolumn{2}{c}{\texttt{Memory (MB)}}   \\
     \cmidrule(lr){1-3}\cmidrule(lr){4-7}\cmidrule(lr){8-9}\cmidrule(lr){10-11}\cmidrule(lr){12-13}\cmidrule(lr){14-15}\cmidrule(lr){16-17}
   &&&  \multicolumn{2}{c}{Compact} & \multicolumn{2}{c}{B\&P}&  Compact & B\&P & Compact & B\&P &Compact & B\&P&Compact & B\&P&Compact & B\&P\\
   \hline

\multirow{18}{*}{20}	&	\multirow{3}{*}{W}	&	2	&	26.03	&(	0	)&\bf	24.22	&\bf(	0	)&	100.00	&\bf	0.00	&	0.00	&	0.00	&\bf	204	&	2385	&	50632	&\bf	1	&\bf	41	&	120	\\
	&		&	5	&	---	&(	5	)&\bf	20.33	&\bf(	0	)&	100.00	&\bf	0.00	&	52.53	&\bf	0.00	&\bf	470	&	2504	&	8992191	&\bf	1	&	15633	&\bf	53	\\
	&		&	10	&	---	&(	5	)&\bf	9.15	&\bf(	0	)&	100.00	&\bf	0.00	&	89.75	&\bf	0.00	&\bf	880	&	2178	&	5306625	&\bf	1	&	13982	&\bf	28	\\
\cline{2-17}	&	\multirow{3}{*}{C}	&	2	&\bf	0.16	&\bf(	0	)&	757.86	&(	4	)&	100.00	&\bf	18.45	&\bf	0.00	&	10.76	&\bf	204	&	45528	&\bf	119	&	13798	&\bf	4	&	962	\\
	&		&	5	&\bf	16.30	&\bf(	0	)&	---	&(	5	)&	100.00	&\bf	32.78	&\bf	0.00	&	21.37	&\bf	470	&	7002	&	13102	&\bf	7867	&\bf	31	&	148	\\
	&		&	10	&\bf	983.51	&\bf(	0	)&	2223.12	&(	3	)&	100.00	&\bf	27.39	&\bf	0.00	&	10.97	&\bf	880	&	3717	&	650055	&\bf	9690	&	1569	&\bf	75	\\
\cline{2-17}	&	\multirow{3}{*}{K}	&	2	&\bf	5.03	&\bf(	0	)&	1648.82	&(	1	)&	100.00	&\bf	7.02	&\bf	0.00	&	1.49	&\bf	204	&	11126	&	6787	&\bf	349	&\bf	11	&	301	\\
	&		&	5	&	4231.23	&(	4	)&\bf	2850.11	&\bf(	2	)&	100.00	&\bf	13.04	&	23.89	&\bf	2.37	&\bf	470	&	5627	&	8109266	&\bf	1830	&	7136	&\bf	117	\\
	&		&	10	&	---	&(	5	)&\bf	1822.27	&\bf(	2	)&	100.00	&\bf	16.01	&	85.45	&\bf	6.61	&\bf	880	&	2925	&	4227203	&\bf	4287	&	11279	&\bf	68	\\
\cline{2-17}	&	\multirow{3}{*}{D}	&	2	&	25.62	&(	0	)&\bf	20.60	&\bf(	0	)&	100.00	&\bf	0.00	&	0.00	&	0.00	&\bf	204	&	2365	&	45501	&\bf	1	&\bf	37	&	119	\\
	&		&	5	&	---	&(	5	)&\bf	15.53	&\bf(	0	)&	100.00	&\bf	0.00	&	50.45	&\bf	0.00	&\bf	470	&	2495	&	9156476	&\bf	1	&	11659	&\bf	52	\\
	&		&	10	&	---	&(	5	)&\bf	17.64	&\bf(	0	)&	100.00	&\bf	0.05	&	87.20	&\bf	0.00	&\bf	880	&	2175	&	5927899	&\bf	2	&	15624	&\bf	28	\\
\cline{2-17}	&	\multirow{3}{*}{S}	&	2	&\bf	25.31	&\bf(	0	)&	28.86	&(	0	)&	100.00	&\bf	0.00	&\bf	0.00	&	0.00	&\bf	204	&	2445	&	42178	&\bf	1	&\bf	37	&	126	\\
	&		&	5	&	---	&(	5	)&\bf	49.89	&\bf(	0	)&	100.00	&\bf	0.03	&	49.72	&\bf	0.00	&\bf	470	&	2491	&	9117840	&\bf	3	&	12277	&\bf	53	\\
	&		&	10	&	---	&(	5	)&\bf	24.45	&\bf(	0	)&	100.00	&\bf	0.25	&	87.53	&\bf	0.00	&\bf	880	&	2183	&	5399238	&\bf	4	&	15145	&\bf	29	\\
\cline{2-17}	&	\multirow{3}{*}{A}	&	2	&\bf	15.72	&\bf(	0	)&	1663.90	&(	1	)&	100.00	&\bf	3.21	&\bf	0.00	&	0.17	&\bf	204	&	6400	&	13975	&\bf	133	&\bf	19	&	288	\\
	&		&	5	&	---	&(	5	)&\bf	1603.15	&\bf(	2	)&	100.00	&\bf	9.93	&	36.09	&\bf	1.74	&\bf	470	&	5066	&	6284524	&\bf	703	&	9720	&\bf	138	\\
	&		&	10	&	---	&(	5	)&\bf	2420.97	&\bf(	1	)&	100.00	&\bf	15.86	&	87.63	&\bf	0.94	&\bf	880	&	2987	&	3478065	&\bf	2125	&	11095	&\bf	60	\\
\cline{1-17}\multirow{18}{*}{30}	&	\multirow{3}{*}{W}	&	2	&	1214.67	&(	0	)&\bf	264.29	&\bf(	0	)&	86.79	&\bf	0.00	&	0.00	&	0.00	&\bf	304	&	3787	&	1814638	&\bf	1	&	559	&\bf	339	\\
	&		&	5	&	---	&(	5	)&\bf	80.27	&\bf(	0	)&	86.91	&\bf	0.00	&	72.61	&\bf	0.00	&\bf	700	&	2914	&	3104800	&\bf	1	&	12866	&\bf	107	\\
	&		&	10	&	---	&(	5	)&\bf	34.14	&\bf(	0	)&	87.86	&\bf	0.00	&	84.94	&\bf	0.00	&\bf	1310	&	2474	&	1626671	&\bf	1	&	18935	&\bf	54	\\
\cline{2-17}	&	\multirow{3}{*}{C}	&	2	&\bf	0.43	&\bf(	0	)&	104.99	&(	4	)&	81.07	&\bf	14.80	&\bf	0.00	&	11.58	&\bf	304	&	45929	&\bf	422	&	1794	&\bf	6	&	643	\\
	&		&	5	&\bf	76.89	&\bf(	0	)&	503.59	&(	4	)&	81.71	&\bf	27.87	&\bf	0.00	&	20.17	&\bf	700	&	8380	&	40315	&\bf	1409	&	165	&\bf	100	\\
	&		&	10	&\bf	1.89	&\bf(	4	)&	---	&(	5	)&	81.93	&\bf	38.05	&	41.89	&\bf	33.92	&\bf	1310	&	4996	&	2709802	&\bf	2978	&	4156	&\bf	69	\\
\cline{2-17}	&	\multirow{3}{*}{K}	&	2	&\bf	54.31	&\bf(	0	)&	744.55	&(	4	)&	85.50	&\bf	6.60	&\bf	0.00	&	6.45	&\bf	304	&	7743	&	62564	&\bf	44	&\bf	58	&	368	\\
	&		&	5	&	---	&(	5	)&\bf	5040.71	&\bf(	4	)&	86.00	&\bf	11.54	&	67.36	&\bf	8.59	&\bf	700	&	5932	&	2838928	&\bf	295	&	15993	&\bf	170	\\
	&		&	10	&	---	&(	5	)&	---	&(	5	)&	87.68	&\bf	17.84	&	83.29	&\bf	13.27	&\bf	1310	&	4040	&	1799130	&\bf	1270	&	12831	&\bf	88	\\
\cline{2-17}	&	\multirow{3}{*}{D}	&	2	&	1247.78	&(	0	)&\bf	222.14	&\bf(	0	)&	86.38	&\bf	0.00	&	0.00	&	0.00	&\bf	304	&	3810	&	1731910	&\bf	1	&	529	&\bf	335	\\
	&		&	5	&	---	&(	5	)&\bf	139.30	&\bf(	0	)&	86.94	&\bf	0.00	&	73.46	&\bf	0.00	&\bf	700	&	2939	&	2824067	&\bf	1	&	12906	&\bf	108	\\
	&		&	10	&	---	&(	5	)&\bf	38.85	&\bf(	0	)&	87.51	&\bf	0.00	&	85.28	&\bf	0.00	&\bf	1310	&	2480	&	1519282	&\bf	1	&	16694	&\bf	55	\\
\cline{2-17}	&	\multirow{3}{*}{S}	&	2	&	1271.98	&(	0	)&\bf	186.78	&\bf(	0	)&	86.69	&\bf	0.00	&\bf	0.00	&	0.00	&\bf	304	&	3699	&	1631950	&\bf	1	&	518	&\bf	324	\\
	&		&	5	&	---	&(	5	)&\bf	540.52	&\bf(	0	)&	87.12	&\bf	0.02	&	73.42	&\bf	0.00	&\bf	700	&	2905	&	3042832	&\bf	3	&	13666	&\bf	109	\\
	&		&	10	&	---	&(	5	)&\bf	94.57	&\bf(	0	)&	87.66	&\bf	0.08	&	84.50	&\bf	0.00	&\bf	1310	&	2465	&	1569091	&\bf	2	&	13710	&\bf	55	\\
\cline{2-17}	&	\multirow{3}{*}{A}	&	2	&\bf	220.34	&\bf(	0	)&	3254.28	&(	3	)&	85.20	&\bf	2.15	&\bf	0.00	&	1.69	&\bf	304	&	4997	&	136052	&\bf	16	&\bf	110	&	403	\\
	&		&	5	&	---	&(	5	)&\bf	1745.13	&\bf(	4	)&	86.59	&\bf	5.37	&	71.23	&\bf	3.06	&\bf	700	&	4216	&	1909608	&\bf	51	&	10892	&\bf	174	\\
	&		&	10	&	---	&(	5	)&\bf	1207.09	&\bf(	4	)&	86.97	&\bf	11.27	&	83.28	&\bf	6.38	&\bf	1310	&	3735	&	1103129	&\bf	465	&	10680	&\bf	106	\\
\cline{1-17}\multirow{18}{*}{40}	&	\multirow{3}{*}{W}	&	2	&	---	&(	5	)&\bf	252.49	&\bf(	0	)&	100.00	&\bf	0.00	&	31.03	&\bf	0.00	&\bf	404	&	6365	&	7312331	&\bf	1	&	5864	&\bf	817	\\
	&		&	5	&	---	&(	5	)&\bf	364.88	&\bf(	0	)&	100.00	&\bf	0.00	&	94.01	&\bf	0.00	&\bf	930	&	4168	&	1925927	&\bf	1	&	11549	&\bf	234	\\
	&		&	10	&	---	&(	5	)&\bf	355.67	&\bf(	0	)&	100.00	&\bf	0.00	&	99.63	&\bf	0.00	&\bf	1740	&	3922	&	862095	&\bf	1	&	8294	&\bf	120	\\
\cline{2-17}	&	\multirow{3}{*}{C}	&	2	&\bf	1.55	&\bf(	0	)&	---	&(	5	)&	100.00	&\bf	23.26	&\bf	0.00	&	22.73	&\bf	404	&	29032	&	1164	&\bf	419	&\bf	9	&	448	\\
	&		&	5	&\bf	594.78	&\bf(	0	)&	---	&(	5	)&	100.00	&\bf	36.17	&\bf	0.00	&	35.69	&\bf	930	&	5983	&	271138	&\bf	92	&	862	&\bf	57	\\
	&		&	10	&\bf	4207.08	&\bf(	1	)&	---	&(	5	)&	100.00	&\bf	37.29	&\bf	11.40	&	35.81	&\bf	1740	&	4780	&	1357589	&\bf	266	&	5116	&\bf	40	\\
\cline{2-17}	&	\multirow{3}{*}{K}	&	2	&\bf	719.29	&\bf(	0	)&	---	&(	5	)&	100.00	&\bf	8.69	&\bf	0.00	&	8.69	&\bf	404	&	5651	&	701127	&\bf	4	&	429	&\bf	412	\\
	&		&	5	&	---	&(	5	)&	---	&(	5	)&	100.00	&\bf	16.32	&	89.79	&\bf	16.31	&\bf	930	&	4724	&	1771444	&\bf	4	&	11673	&\bf	171	\\
	&		&	10	&	---	&(	5	)&	---	&(	5	)&	100.00	&\bf	16.50	&	99.43	&\bf	16.47	&\bf	1740	&	4836	&	700210	&\bf	44	&	13117	&\bf	100	\\
\cline{2-17}	&	\multirow{3}{*}{D}	&	2	&	---	&(	5	)&\bf	259.07	&\bf(	0	)&	100.00	&\bf	0.00	&	30.59	&\bf	0.00	&\bf	404	&	6694	&	6885185	&\bf	1	&	7289	&\bf	865	\\
	&		&	5	&	---	&(	5	)&\bf	485.52	&\bf(	0	)&	100.00	&\bf	0.00	&	94.27	&\bf	0.00	&\bf	930	&	4240	&	2040736	&\bf	1	&	9260	&\bf	238	\\
	&		&	10	&	---	&(	5	)&\bf	473.11	&\bf(	0	)&	100.00	&\bf	0.01	&	99.90	&\bf	0.00	&\bf	1740	&	3891	&	692838	&\bf	1	&	10352	&\bf	119	\\
\cline{2-17}	&	\multirow{3}{*}{S}	&	2	&	---	&(	5	)&\bf	691.25	&\bf(	0	)&	100.00	&\bf	0.02	&	29.89	&\bf	0.00	&\bf	404	&	6224	&	6075474	&\bf	1	&	5159	&\bf	806	\\
	&		&	5	&	---	&(	5	)&\bf	244.29	&\bf(	0	)&	100.00	&\bf	0.00	&	94.33	&\bf	0.00	&\bf	930	&	4120	&	2196020	&\bf	1	&	11994	&\bf	230	\\
	&		&	10	&	---	&(	5	)&\bf	1223.93	&\bf(	0	)&	100.00	&\bf	0.07	&	99.90	&\bf	0.00	&\bf	1740	&	3910	&	925128	&\bf	3	&	11056	&\bf	120	\\
\cline{2-17}	&	\multirow{3}{*}{A}	&	2	&\bf	3347.72	&\bf(	4	)&	---	&(	5	)&	100.00	&\bf	4.09	&	9.28	&\bf	4.09	&\bf	404	&	5760	&	2823594	&\bf	2	&	1885	&\bf	763	\\
	&		&	5	&	---	&(	5	)&	---	&(	5	)&	100.00	&\bf	9.56	&	91.43	&\bf	9.08	&\bf	930	&	4410	&	1346465	&\bf	5	&	6365	&\bf	258	\\
	&		&	10	&	---	&(	5	)&	---	&(	5	)&	100.00	&\bf	9.54	&	99.86	&\bf	8.60	&\bf	1740	&	4404	&	677970	&\bf	41	&	9029	&\bf	141	\\
\cline{1-17}\multirow{18}{*}{45}	&	\multirow{3}{*}{W}	&	2	&	---	&(	5	)&\bf	388.20	&\bf(	0	)&	100.00	&\bf	0.00	&	41.73	&\bf	0.00	&\bf	454	&	9960	&	4536228	&\bf	1	&	6272	&\bf	1570	\\
	&		&	5	&	---	&(	5	)&\bf	207.18	&\bf(	0	)&	100.00	&\bf	0.00	&	96.23	&\bf	0.00	&\bf	1045	&	4741	&	1544401	&\bf	1	&	10363	&\bf	360	\\
	&		&	10	&	---	&(	5	)&\bf	282.84	&\bf(	0	)&	100.00	&\bf	0.00	&	100.00	&\bf	0.00	&\bf	1955	&	4318	&	494940	&\bf	1	&	10023	&\bf	172	\\
\cline{2-17}	&	\multirow{3}{*}{C}	&	2	&\bf	2.49	&\bf(	0	)&	---	&(	5	)&	100.00	&\bf	22.33	&\bf	0.00	&	22.32	&\bf	454	&	4517	&	1618	&\bf	6	&\bf	13	&	67	\\
	&		&	5	&\bf	398.50	&\bf(	0	)&	---	&(	5	)&	100.00	&\bf	35.75	&\bf	0.00	&	35.62	&\bf	1045	&	4850	&	162356	&\bf	10	&	453	&\bf	46	\\
	&		&	10	&	---	&(	5	)&	---	&(	5	)&	100.00	&\bf	35.00	&	76.29	&\bf	34.76	&\bf	1955	&	4710	&	1628650	&\bf	35	&	5534	&\bf	36	\\
\cline{2-17}	&	\multirow{3}{*}{K}	&	2	&\bf	5287.03	&\bf(	2	)&	---	&(	5	)&	100.00	&\bf	12.13	&\bf	5.66	&	12.12	&\bf	454	&	6140	&	4935475	&\bf	3	&	2537	&\bf	535	\\
	&		&	5	&	---	&(	5	)&	---	&(	5	)&	100.00	&\bf	17.38	&	93.84	&\bf	17.36	&\bf	1045	&	5116	&	2086094	&\bf	4	&	9338	&\bf	226	\\
	&		&	10	&	---	&(	5	)&	---	&(	5	)&	100.00	&\bf	14.67	&	100.00	&\bf	14.67	&\bf	1955	&	4781	&	618710	&\bf	5	&	11268	&\bf	120	\\
\cline{2-17}	&	\multirow{3}{*}{D}	&	2	&	---	&(	5	)&\bf	358.37	&\bf(	0	)&	100.00	&\bf	0.00	&	40.50	&\bf	0.00	&\bf	454	&	9220	&	4961123	&\bf	1	&	5740	&\bf	1483	\\
	&		&	5	&	---	&(	5	)&\bf	207.01	&\bf(	0	)&	100.00	&\bf	0.00	&	96.12	&\bf	0.00	&\bf	1045	&	4756	&	1837767	&\bf	1	&	8596	&\bf	363	\\
	&		&	10	&	---	&(	5	)&\bf	483.48	&\bf(	0	)&	100.00	&\bf	0.00	&	100.00	&\bf	0.00	&\bf	1955	&	4310	&	486390	&\bf	1	&	9703	&\bf	172	\\
\cline{2-17}	&	\multirow{3}{*}{S}	&	2	&	---	&(	5	)&\bf	370.88	&\bf(	0	)&	100.00	&\bf	0.00	&	40.04	&\bf	0.00	&\bf	454	&	9566	&	4705056	&\bf	1	&	5430	&\bf	1507	\\
	&		&	5	&	---	&(	5	)&\bf	1332.07	&\bf(	0	)&	100.00	&\bf	0.20	&	95.75	&\bf	0.00	&\bf	1045	&	4809	&	2080272	&\bf	4	&	10194	&\bf	370	\\
	&		&	10	&	---	&(	5	)&\bf	1487.71	&\bf(	0	)&	100.00	&\bf	0.02	&	99.90	&\bf	0.00	&\bf	1955	&	4315	&	659990	&\bf	3	&	10857	&\bf	173	\\
\cline{2-17}	&	\multirow{3}{*}{A}	&	2	&	---	&(	5	)&	---	&(	5	)&	100.00	&\bf	6.72	&	25.73	&\bf	6.26	&\bf	454	&	6957	&	2069919	&\bf	4	&	3364	&\bf	1078	\\
	&		&	5	&	---	&(	5	)&	---	&(	5	)&	100.00	&\bf	11.97	&	91.39	&\bf	11.12	&\bf	1045	&	4895	&	1615563	&\bf	6	&	3655	&\bf	373	\\
	&		&	10	&	---	&(	5	)&	---	&(	5	)&	100.00	&\bf	7.59	&	100.00	&\bf	7.46	&\bf	1955	&	4486	&	483721	&\bf	6	&	8786	&\bf	172	\\
\cline{1-17}\multirow{18}{*}{50}	&	\multirow{3}{*}{W}	&	2	&	---	&(	1	)&\bf	456.01	&\bf(	0	)&	100.00	&\bf	0.00	&	48.21	&\bf	0.00	&\bf	504	&	10414	&	4777385	&\bf	1	&	6240	&\bf	2007	\\
	&		&	5	&	---	&(	1	)&\bf	615.74	&\bf(	0	)&	100.00	&\bf	0.03	&	96.50	&\bf	0.00	&\bf	1160	&	5458	&	1930033	&\bf	3	&	6233	&\bf	470	\\
	&		&	10	&	---	&(	1	)&\bf	204.96	&\bf(	0	)&	100.00	&\bf	0.00	&	100.00	&\bf	0.00	&\bf	2170	&	4971	&	355938	&\bf	1	&	7946	&\bf	232	\\
\cline{2-17}	&	\multirow{3}{*}{C}	&	2	&\bf	4.78	&\bf(	0	)&	---	&(	1	)&	100.00	&\bf	22.61	&\bf	0.00	&	22.61	&\bf	504	&	4342	&	2935	&\bf	2	&\bf	17	&	59	\\
	&		&	5	&	---	&(	1	)&	---	&(	1	)&	100.00	&\bf	34.42	&	67.17	&\bf	34.37	&\bf	1160	&	5142	&	2940681	&\bf	4	&	5728	&\bf	51	\\
	&		&	10	&	---	&(	1	)&	---	&(	1	)&	100.00	&\bf	40.27	&	83.64	&\bf	40.21	&\bf	2170	&	5198	&	1333733	&\bf	18	&	3970	&\bf	41	\\
\cline{2-17}	&	\multirow{3}{*}{K}	&	2	&	---	&(	1	)&	---	&(	1	)&	100.00	&\bf	11.24	&	18.49	&\bf	11.24	&\bf	504	&	6665	&	4998109	&\bf	3	&	3635	&\bf	673	\\
	&		&	5	&	---	&(	1	)&	---	&(	1	)&	100.00	&\bf	17.43	&	91.41	&\bf	17.43	&\bf	1160	&	5674	&	3154623	&\bf	3	&	5249	&\bf	284	\\
	&		&	10	&	---	&(	1	)&	---	&(	1	)&	100.00	&\bf	14.01	&	100.00	&\bf	13.91	&\bf	2170	&	5463	&	567705	&\bf	7	&	10827	&\bf	156	\\
\cline{2-17}	&	\multirow{3}{*}{D}	&	2	&	---	&(	1	)&\bf	405.36	&\bf(	0	)&	100.00	&\bf	0.00	&	47.33	&\bf	0.00	&\bf	504	&	9675	&	4367022	&\bf	1	&	6329	&\bf	1789	\\
	&		&	5	&	---	&(	1	)&\bf	209.77	&\bf(	0	)&	100.00	&\bf	0.00	&	94.73	&\bf	0.00	&\bf	1160	&	5302	&	2598004	&\bf	1	&	4692	&\bf	469	\\
	&		&	10	&	---	&(	1	)&\bf	605.10	&\bf(	0	)&	100.00	&\bf	0.00	&	100.00	&\bf	0.00	&\bf	2170	&	5000	&	548086	&\bf	1	&	7504	&\bf	233	\\
\cline{2-17}	&	\multirow{3}{*}{S}	&	2	&	---	&(	1	)&\bf	451.37	&\bf(	0	)&	100.00	&\bf	0.00	&	48.39	&\bf	0.00	&\bf	504	&	10286	&	3512436	&\bf	1	&	6039	&\bf	1986	\\
	&		&	5	&	---	&(	1	)&\bf	815.35	&\bf(	0	)&	100.00	&\bf	0.07	&	96.27	&\bf	0.00	&\bf	1160	&	5236	&	3055646	&\bf	3	&	6241	&\bf	459	\\
	&		&	10	&	---	&(	1	)&\bf	204.63	&\bf(	0	)&	100.00	&\bf	0.01	&	100.00	&\bf	0.01	&\bf	2170	&	4900	&	386889	&\bf	1	&	9984	&\bf	231	\\
\cline{2-17}	&	\multirow{3}{*}{A}	&	2	&	---	&(	1	)&	---	&(	1	)&	100.00	&\bf	5.95	&	32.08	&\bf	5.36	&\bf	504	&	8326	&	1790159	&\bf	4	&	2035	&\bf	1513	\\
	&		&	5	&	---	&(	1	)&	---	&(	1	)&	100.00	&\bf	12.91	&	91.62	&\bf	11.92	&\bf	1160	&	5559	&	1874166	&\bf	7	&	3311	&\bf	483	\\
	&		&	10	&	---	&(	1	)&	---	&(	1	)&	100.00	&\bf	7.13	&	100.00	&\bf	6.88	&\bf	2170	&	5089	&	462300	&\bf	9	&	10198	&\bf	231	\\
\cline{1-17}\multicolumn{3}{c}{\bf Total Average:} 					&	673.68	&(	267	)&\bf	557.40	&\bf(	157	)&	96.65	&\bf	8.44	&	53.08	&\bf	6.79	&\bf	886	&	6179	&	2347788	&\bf	663	&	7186	&\bf	310	\\

\hline
\end{tabular}%
\end{adjustbox}
   \caption{Results for \cite{EWC74} instances for $\ell_2$-norm \label{tab:EilonL2}}}
\end{table}%

 \begin{table}
  \centering{
  \begin{adjustbox}{max width=0.9\textwidth}
  \begin{tabular}{cccrcrcrrrrrrrrrr}\hline
   $n$ & \texttt{type} &$p$ & \multicolumn{4}{c}{\texttt{Time (\#Unsolved)}} &  \multicolumn{2}{c}{\texttt{GAProot(\%)}} &  \multicolumn{2}{c}{\texttt{GAP(\%)}}  & \multicolumn{2}{c}{\texttt{Vars}} &  \multicolumn{2}{c}{\texttt{Nodes}}  &  \multicolumn{2}{c}{\texttt{Memory (MB)}}   \\
     \cmidrule(lr){1-3}\cmidrule(lr){4-7}\cmidrule(lr){8-9}\cmidrule(lr){10-11}\cmidrule(lr){12-13}\cmidrule(lr){14-15}\cmidrule(lr){16-17}
   &&&  \multicolumn{2}{c}{Compact} & \multicolumn{2}{c}{B\&P}&  Compact & B\&P & Compact & B\&P &Compact & B\&P&Compact & B\&P&Compact & B\&P\\
   \hline

\multirow{18}{*}{20}	&	\multirow{3}{*}{W}	&	2	&\bf	151.03	&\bf(	0	)&	743.17	&(	0	)&	100.00	&\bf	0.00	&	0.00	&	0.00	&	524	&\bf	471	&	56126	&\bf	1	&	418	&\bf	24	\\
	&		&	5	&	---	&(	5	)&\bf	267.56	&\bf(	0	)&	100.00	&\bf	0.00	&	62.95	&\bf	0.00	&	1270	&\bf	415	&	2700255	&\bf	1	&	13112	&\bf	14	\\
	&		&	10	&	---	&(	5	)&\bf	133.91	&\bf(	0	)&	100.00	&\bf	0.00	&	98.54	&\bf	0.00	&	2480	&\bf	406	&	1259405	&\bf	1	&	13418	&\bf	13	\\
\cline{2-17}	&	\multirow{3}{*}{C}	&	2	&\bf	0.91	&\bf(	0	)&	---	&(	5	)&	100.00	&\bf	30.56	&\bf	0.00	&	29.19	&\bf	524	&	9181	&	507	&\bf	417	&\bf	11	&	121	\\
	&		&	5	&\bf	71.86	&\bf(	0	)&	---	&(	5	)&	100.00	&\bf	35.58	&\bf	0.00	&	30.56	&\bf	1270	&	3061	&	17235	&\bf	1473	&	62	&\bf	32	\\
	&		&	10	&\bf	2142.39	&\bf(	0	)&	---	&(	5	)&	100.00	&\bf	47.17	&\bf	0.00	&	47.00	&	2480	&\bf	912	&	261767	&\bf	553	&	879	&\bf	8	\\
\cline{2-17}	&	\multirow{3}{*}{K}	&	2	&\bf	16.05	&\bf(	0	)&	5363.15	&(	2	)&	100.00	&\bf	9.14	&\bf	0.00	&	2.07	&\bf	524	&	12397	&	7592	&\bf	424	&\bf	40	&	352	\\
	&		&	5	&	---	&(	5	)&\bf	2946.11	&\bf(	3	)&	100.00	&\bf	14.87	&	38.13	&\bf	5.02	&\bf	1270	&	3466	&	2839547	&\bf	1556	&	8104	&\bf	75	\\
	&		&	10	&	---	&(	5	)&\bf	3206.51	&\bf(	2	)&	100.00	&\bf	8.25	&	93.62	&\bf	5.18	&	2480	&\bf	1348	&	1021924	&\bf	2630	&	6013	&\bf	23	\\
\cline{2-17}	&	\multirow{3}{*}{D}	&	2	&\bf	125.29	&\bf(	0	)&	810.43	&(	0	)&	100.00	&\bf	0.00	&	0.00	&	0.00	&	524	&\bf	498	&	53138	&\bf	1	&	300	&\bf	26	\\
	&		&	5	&	---	&(	5	)&\bf	230.99	&\bf(	0	)&	100.00	&\bf	0.00	&	60.85	&\bf	0.00	&	1270	&\bf	382	&	2732256	&\bf	1	&	12124	&\bf	11	\\
	&		&	10	&	---	&(	5	)&\bf	112.30	&\bf(	0	)&	100.00	&\bf	0.09	&	99.24	&\bf	0.00	&	2480	&\bf	377	&	1083407	&\bf	2	&	11367	&\bf	10	\\
\cline{2-17}	&	\multirow{3}{*}{S}	&	2	&\bf	104.35	&\bf(	0	)&	973.36	&(	0	)&	100.00	&\bf	0.05	&	0.00	&	0.00	&\bf	524	&	530	&	45754	&\bf	1	&	239	&\bf	27	\\
	&		&	5	&	---	&(	5	)&\bf	255.50	&\bf(	0	)&	100.00	&\bf	0.11	&	64.54	&\bf	0.00	&	1270	&\bf	398	&	2470849	&\bf	2	&	10763	&\bf	12	\\
	&		&	10	&	---	&(	5	)&\bf	125.04	&\bf(	0	)&	100.00	&\bf	0.26	&	99.51	&\bf	0.00	&	2480	&\bf	378	&	1131966	&\bf	4	&	11380	&\bf	10	\\
\cline{2-17}	&	\multirow{3}{*}{A}	&	2	&\bf	33.54	&\bf(	0	)&	2705.24	&(	1	)&	100.00	&\bf	4.22	&\bf	0.00	&	0.15	&\bf	524	&	3049	&	11374	&\bf	71	&\bf	74	&	150	\\
	&		&	5	&	---	&(	5	)&\bf	1648.71	&\bf(	2	)&	100.00	&\bf	8.18	&	42.12	&\bf	1.90	&\bf	1270	&	2262	&	2117999	&\bf	329	&	6961	&\bf	64	\\
	&		&	10	&	---	&(	5	)&\bf	4194.46	&\bf(	1	)&	100.00	&\bf	13.35	&	99.33	&\bf	1.41	&	2480	&\bf	1216	&	842028	&\bf	2054	&	8107	&\bf	29	\\
\cline{1-17}\multirow{18}{*}{30}	&	\multirow{3}{*}{W}	&	2	&	2404.31	&(	4	)&\bf	435.00	&\bf(	4	)&	86.64	&\bf	14.53	&\bf	12.54	&	14.53	&\bf	784	&	843	&	2130171	&\bf	1	&	10112	&\bf	65	\\
	&		&	5	&	---	&(	5	)&\bf	3934.80	&\bf(	2	)&	86.35	&\bf	5.57	&	77.53	&\bf	5.57	&	1900	&\bf	860	&	830494	&\bf	1	&	12706	&\bf	48	\\
	&		&	10	&	---	&(	5	)&\bf	2479.98	&\bf(	0	)&	86.89	&\bf	0.03	&	86.14	&\bf	0.00	&	3710	&\bf	816	&	306492	&\bf	3	&	4764	&\bf	37	\\
\cline{2-17}	&	\multirow{3}{*}{C}	&	2	&\bf	2.30	&\bf(	0	)&	1631.39	&(	4	)&	81.07	&\bf	29.54	&\bf	0.00	&	29.53	&\bf	784	&	7137	&\bf	675	&	7559	&\bf	17	&	94	\\
	&		&	5	&\bf	571.24	&\bf(	0	)&	2360.18	&(	4	)&	81.71	&\bf	48.42	&\bf	0.00	&	46.18	&\bf	1900	&	2388	&	71418	&\bf	386	&	212	&\bf	27	\\
	&		&	10	&\bf	8.07	&\bf(	4	)&	1801.17	&(	4	)&	81.93	&\bf	48.01	&	71.16	&\bf	47.51	&	3710	&\bf	1236	&	316323	&\bf	33	&	1003	&\bf	10	\\
\cline{2-17}	&	\multirow{3}{*}{K}	&	2	&\bf	211.31	&\bf(	0	)&	17.36	&(	4	)&	85.39	&\bf	24.11	&\bf	0.00	&	23.96	&\bf	784	&	1808	&	86793	&\bf	7	&	303	&\bf	81	\\
	&		&	5	&	---	&(	5	)&\bf	27.59	&\bf(	4	)&	86.11	&\bf	14.05	&	63.23	&\bf	13.64	&	1900	&\bf	1707	&	1143862	&\bf	12	&	9716	&\bf	52	\\
	&		&	10	&	---	&(	5	)&\bf	34.09	&\bf(	4	)&	85.92	&\bf	22.77	&	85.05	&\bf	22.13	&	3710	&\bf	1409	&	365197	&\bf	168	&	4413	&\bf	35	\\
\cline{2-17}	&	\multirow{3}{*}{D}	&	2	&	4446.26	&(	3	)&\bf	3777.09	&\bf(	3	)&	86.24	&\bf	18.03	&\bf	10.14	&	18.03	&\bf	784	&	946	&	2197143	&\bf	1	&	9689	&\bf	71	\\
	&		&	5	&	---	&(	5	)&\bf	3574.66	&\bf(	2	)&	86.38	&\bf	9.19	&	74.99	&\bf	9.19	&	1900	&\bf	853	&	958686	&\bf	1	&	12995	&\bf	45	\\
	&		&	10	&	---	&(	5	)&\bf	2472.72	&\bf(	0	)&	86.60	&\bf	0.06	&	85.43	&\bf	0.00	&	3710	&\bf	796	&	375899	&\bf	3	&	5968	&\bf	35	\\
\cline{2-17}	&	\multirow{3}{*}{S}	&	2	&\bf	3256.59	&\bf(	3	)&	305.27	&(	4	)&	86.54	&\bf	17.51	&\bf	9.82	&	17.51	&\bf	784	&	850	&	1889060	&\bf	1	&	8632	&\bf	70	\\
	&		&	5	&	---	&(	5	)&\bf	3697.87	&\bf(	1	)&	86.19	&\bf	1.29	&	78.49	&\bf	1.23	&	1900	&\bf	760	&	873284	&\bf	2	&	12851	&\bf	40	\\
	&		&	10	&	---	&(	5	)&\bf	2525.29	&\bf(	0	)&	87.92	&\bf	0.12	&	87.51	&\bf	0.00	&	3710	&\bf	791	&	351574	&\bf	4	&	4490	&\bf	32	\\
\cline{2-17}	&	\multirow{3}{*}{A}	&	2	&\bf	1365.87	&\bf(	0	)&	61.55	&(	4	)&	85.11	&\bf	21.30	&\bf	0.00	&	21.04	&\bf	784	&	1141	&	327100	&\bf	3	&	1255	&\bf	96	\\
	&		&	5	&	---	&(	5	)&\bf	11.83	&\bf(	4	)&	85.51	&\bf	7.43	&	72.18	&\bf	6.75	&	1900	&\bf	1251	&	643473	&\bf	8	&	9003	&\bf	67	\\
	&		&	10	&	---	&(	5	)&\bf	100.35	&\bf(	4	)&	85.84	&\bf	15.80	&	84.79	&\bf	13.75	&	3710	&\bf	1385	&	306764	&\bf	101	&	4475	&\bf	51	\\
\cline{1-17}\multirow{18}{*}{40}	&	\multirow{3}{*}{W}	&	2	&	---	&(	5	)&	---	&(	5	)&	100.00	&\bf	26.96	&	51.29	&\bf	26.96	&\bf	1044	&	1248	&	1344963	&\bf	1	&	16445	&\bf	121	\\
	&		&	5	&	---	&(	5	)&	---	&(	5	)&	100.00	&\bf	83.14	&	95.28	&\bf	83.14	&	2530	&\bf	1267	&	482105	&\bf	1	&	5041	&\bf	104	\\
	&		&	10	&	---	&(	5	)&	---	&(	5	)&	100.00	&\bf	92.94	&	100.00	&\bf	92.94	&	4940	&\bf	1350	&	188692	&\bf	1	&	3381	&\bf	85	\\
\cline{2-17}	&	\multirow{3}{*}{C}	&	2	&\bf	6.66	&\bf(	0	)&	---	&(	5	)&	100.00	&\bf	38.00	&\bf	0.00	&	37.95	&\bf	1044	&	2502	&	1762	&\bf	5	&\bf	22	&	37	\\
	&		&	5	&\bf	2658.74	&\bf(	2	)&	---	&(	5	)&	100.00	&\bf	81.92	&\bf	39.02	&	81.92	&	2530	&\bf	2235	&	434526	&\bf	1	&	3523	&\bf	25	\\
	&		&	10	&	---	&(	5	)&	---	&(	5	)&	100.00	&\bf	75.34	&	100.00	&\bf	75.34	&	4940	&\bf	1921	&	349858	&\bf	12	&	1348	&\bf	18	\\
\cline{2-17}	&	\multirow{3}{*}{K}	&	2	&\bf	4990.55	&\bf(	1	)&	---	&(	5	)&	100.00	&\bf	50.04	&\bf	3.02	&	50.04	&\bf	1044	&	2202	&	1651655	&\bf	1	&	5136	&\bf	113	\\
	&		&	5	&	---	&(	5	)&	---	&(	5	)&	100.00	&\bf	74.12	&	97.64	&\bf	74.12	&	2530	&\bf	2018	&	479613	&\bf	1	&	7167	&\bf	86	\\
	&		&	10	&	---	&(	5	)&	---	&(	5	)&	100.00	&\bf	69.68	&	100.00	&\bf	69.68	&	4940	&\bf	1934	&	210396	&\bf	3	&	2563	&\bf	72	\\
\cline{2-17}	&	\multirow{3}{*}{D}	&	2	&	---	&(	5	)&	---	&(	5	)&	100.00	&\bf	20.20	&	42.95	&\bf	20.20	&\bf	1044	&	1247	&	1505294	&\bf	1	&	12264	&\bf	125	\\
	&		&	5	&	---	&(	5	)&	---	&(	5	)&	100.00	&\bf	89.82	&	98.91	&\bf	89.82	&	2530	&\bf	1248	&	463628	&\bf	1	&	8186	&\bf	98	\\
	&		&	10	&	---	&(	5	)&	---	&(	5	)&	100.00	&\bf	90.59	&	100.00	&\bf	90.59	&	4940	&\bf	1350	&	153036	&\bf	1	&	5215	&\bf	87	\\
\cline{2-17}	&	\multirow{3}{*}{S}	&	2	&	---	&(	5	)&	---	&(	5	)&	100.00	&\bf	25.20	&	43.94	&\bf	25.20	&\bf	1044	&	1323	&	1360254	&\bf	1	&	12167	&\bf	126	\\
	&		&	5	&	---	&(	5	)&	---	&(	5	)&	100.00	&\bf	70.38	&	96.48	&\bf	70.38	&	2530	&\bf	1294	&	462667	&\bf	1	&	6780	&\bf	91	\\
	&		&	10	&	---	&(	5	)&	---	&(	5	)&	100.00	&\bf	84.09	&	100.00	&\bf	84.09	&	4940	&\bf	1334	&	139121	&\bf	1	&	6489	&\bf	81	\\
\cline{2-17}	&	\multirow{3}{*}{A}	&	2	&	---	&(	5	)&	---	&(	5	)&	100.00	&\bf	24.46	&\bf	22.07	&	24.46	&\bf	1044	&	1810	&	1460678	&\bf	1	&	5826	&\bf	191	\\
	&		&	5	&	---	&(	5	)&	---	&(	5	)&	100.00	&\bf	62.67	&	99.14	&\bf	62.67	&	2530	&\bf	1646	&	345643	&\bf	1	&	7108	&\bf	134	\\
	&		&	10	&	---	&(	5	)&	---	&(	5	)&	100.00	&\bf	85.25	&	100.00	&\bf	85.25	&	4940	&\bf	1767	&	155190	&\bf	1	&	3426	&\bf	113	\\
\cline{1-17}\multirow{18}{*}{45}	&	\multirow{3}{*}{W}	&	2	&	---	&(	5	)&	---	&(	5	)&	100.00	&\bf	26.30	&	54.64	&\bf	26.30	&\bf	1174	&	1398	&	1133697	&\bf	1	&	11416	&\bf	134	\\
	&		&	5	&	---	&(	5	)&	---	&(	5	)&	100.00	&\bf	92.43	&	99.51	&\bf	92.43	&	2845	&\bf	1413	&	401143	&\bf	1	&	5947	&\bf	113	\\
	&		&	10	&	---	&(	5	)&	---	&(	5	)&	100.00	&	100.00	&	100.00	&	100.00	&	5555	&\bf	1469	&	153297	&\bf	1	&	2343	&\bf	104	\\
\cline{2-17}	&	\multirow{3}{*}{C}	&	2	&\bf	6.40	&\bf(	0	)&	---	&(	5	)&	100.00	&\bf	46.67	&\bf	0.00	&	46.67	&\bf	1174	&	2326	&	1132	&\bf	2	&\bf	24	&	34	\\
	&		&	5	&\bf	5003.88	&\bf(	3	)&	---	&(	5	)&	100.00	&\bf	75.71	&\bf	60.00	&	75.71	&	2845	&\bf	2532	&	500043	&\bf	1	&	5891	&\bf	29	\\
	&		&	10	&	---	&(	5	)&	---	&(	5	)&	100.00	&\bf	75.66	&	100.00	&\bf	75.64	&	5555	&\bf	2236	&	171358	&\bf	7	&	722	&\bf	22	\\
\cline{2-17}	&	\multirow{3}{*}{K}	&	2	&	---	&(	5	)&	---	&(	5	)&	100.00	&\bf	44.43	&\bf	28.39	&	44.43	&\bf	1174	&	2477	&	1611616	&\bf	1	&	11832	&\bf	146	\\
	&		&	5	&	---	&(	5	)&	---	&(	5	)&	100.00	&\bf	80.52	&	99.07	&\bf	80.52	&	2845	&\bf	2482	&	426670	&\bf	1	&	6214	&\bf	157	\\
	&		&	10	&	---	&(	5	)&	---	&(	5	)&	100.00	&\bf	84.97	&	100.00	&\bf	84.97	&	5555	&\bf	2165	&	160867	&\bf	1	&	1944	&\bf	90	\\
\cline{2-17}	&	\multirow{3}{*}{D}	&	2	&	---	&(	5	)&	---	&(	5	)&	100.00	&\bf	24.77	&	54.89	&\bf	24.77	&\bf	1174	&	1406	&	1171833	&\bf	1	&	12252	&\bf	135	\\
	&		&	5	&	---	&(	5	)&	---	&(	5	)&	100.00	&\bf	98.15	&	98.67	&\bf	98.15	&	2845	&\bf	1412	&	424560	&\bf	1	&	4188	&\bf	113	\\
	&		&	10	&	---	&(	5	)&	---	&(	5	)&	100.00	&	100.00	&	100.00	&	100.00	&	5555	&\bf	1466	&	152963	&\bf	1	&	2822	&\bf	104	\\
\cline{2-17}	&	\multirow{3}{*}{S}	&	2	&	---	&(	5	)&	---	&(	5	)&	100.00	&\bf	23.69	&	56.00	&\bf	23.69	&\bf	1174	&	1479	&	1069851	&\bf	1	&	11291	&\bf	141	\\
	&		&	5	&	---	&(	5	)&	---	&(	5	)&	100.00	&\bf	72.98	&	99.95	&\bf	72.98	&	2845	&\bf	1450	&	514193	&\bf	1	&	5809	&\bf	107	\\
	&		&	10	&	---	&(	5	)&	---	&(	5	)&	100.00	&	100.00	&	100.00	&	100.00	&	5555	&\bf	1467	&	159205	&\bf	1	&	2847	&\bf	101	\\
\cline{2-17}	&	\multirow{3}{*}{A}	&	2	&	---	&(	5	)&	---	&(	5	)&	100.00	&\bf	24.73	&	37.56	&\bf	24.73	&\bf	1174	&	1731	&	695531	&\bf	1	&	6228	&\bf	178	\\
	&		&	5	&	---	&(	5	)&	---	&(	5	)&	100.00	&\bf	65.14	&	97.05	&\bf	65.14	&	2845	&\bf	1800	&	295359	&\bf	1	&	4504	&\bf	160	\\
	&		&	10	&	---	&(	5	)&	---	&(	5	)&	100.00	&\bf	95.05	&	100.00	&\bf	95.05	&	5555	&\bf	2027	&	131122	&\bf	1	&	1477	&\bf	172	\\
\cline{1-17}\multirow{18}{*}{50}	&	\multirow{3}{*}{W}	&	2	&	---	&(	1	)&	---	&(	1	)&	100.00	&\bf	75.57	&\bf	61.16	&	75.57	&\bf	1304	&	1616	&	1047768	&\bf	1	&	11081	&\bf	161	\\
	&		&	5	&	---	&(	1	)&	---	&(	1	)&	100.00	&\bf	99.83	&	100.00	&\bf	99.83	&	3160	&\bf	1623	&	406553	&\bf	1	&	2969	&\bf	141	\\
	&		&	10	&	---	&(	1	)&	---	&(	1	)&	100.00	&	100.00	&	100.00	&	100.00	&	6170	&\bf	1709	&	186637	&\bf	1	&	934	&\bf	128	\\
\cline{2-17}	&	\multirow{3}{*}{C}	&	2	&\bf	11.71	&\bf(	0	)&	---	&(	1	)&	100.00	&\bf	62.36	&\bf	0.00	&	62.36	&\bf	1304	&	2575	&	2112	&\bf	1	&	38	&\bf	37	\\
	&		&	5	&	---	&(	1	)&	---	&(	1	)&	100.00	&\bf	90.14	&\bf	87.76	&	90.14	&	3160	&\bf	2577	&	540242	&\bf	1	&	1441	&\bf	34	\\
	&		&	10	&	---	&(	1	)&	---	&(	1	)&	100.00	&\bf	99.12	&	100.00	&\bf	99.12	&	6170	&\bf	2587	&	112265	&\bf	1	&	572	&\bf	26	\\
\cline{2-17}	&	\multirow{3}{*}{K}	&	2	&	---	&(	1	)&	---	&(	1	)&	100.00	&\bf	36.29	&	46.27	&\bf	36.29	&\bf	1304	&	2810	&	957166	&\bf	1	&	10628	&\bf	164	\\
	&		&	5	&	---	&(	1	)&	---	&(	1	)&	100.00	&\bf	91.13	&	98.89	&\bf	91.13	&	3160	&\bf	2450	&	434558	&\bf	1	&	4069	&\bf	127	\\
	&		&	10	&	---	&(	1	)&	---	&(	1	)&	100.00	&\bf	90.03	&	100.00	&\bf	90.03	&	6170	&\bf	2410	&	142538	&\bf	1	&	1231	&\bf	108	\\
\cline{2-17}	&	\multirow{3}{*}{D}	&	2	&	---	&(	1	)&	---	&(	1	)&	100.00	&\bf	31.05	&	63.12	&\bf	31.05	&\bf	1304	&	1666	&	878351	&\bf	1	&	10528	&\bf	162	\\
	&		&	5	&	---	&(	1	)&	---	&(	1	)&	100.00	&\bf	74.10	&	100.00	&\bf	74.10	&	3160	&\bf	1695	&	285611	&\bf	1	&	4410	&\bf	134	\\
	&		&	10	&	---	&(	1	)&	---	&(	1	)&	100.00	&	100.00	&	100.00	&	100.00	&	6170	&\bf	1777	&	83410	&\bf	1	&	545	&\bf	129	\\
\cline{2-17}	&	\multirow{3}{*}{S}	&	2	&	---	&(	1	)&	---	&(	1	)&	100.00	&\bf	36.71	&	60.33	&\bf	36.71	&\bf	1304	&	1577	&	983547	&\bf	1	&	10718	&\bf	148	\\
	&		&	5	&	---	&(	1	)&	---	&(	1	)&	100.00	&\bf	94.34	&	99.55	&\bf	94.34	&	3160	&\bf	1635	&	583530	&\bf	1	&	2230	&\bf	129	\\
	&		&	10	&	---	&(	1	)&	---	&(	1	)&	100.00	&	100.00	&	100.00	&	100.00	&	6170	&\bf	1737	&	153870	&\bf	1	&	775	&\bf	135	\\
\cline{2-17}	&	\multirow{3}{*}{A}	&	2	&	---	&(	1	)&	---	&(	1	)&	100.00	&\bf	24.67	&	48.59	&\bf	24.67	&\bf	1304	&	2258	&	514472	&\bf	1	&	5470	&\bf	226	\\
	&		&	5	&	---	&(	1	)&	---	&(	1	)&	100.00	&\bf	75.58	&	100.00	&\bf	75.58	&	3160	&\bf	1998	&	403520	&\bf	1	&	2368	&\bf	182	\\
	&		&	10	&	---	&(	1	)&	---	&(	1	)&	100.00	&\bf	91.23	&	100.00	&\bf	91.23	&	6170	&\bf	2051	&	117660	&\bf	1	&	1968	&\bf	170	\\
\cline{1-17}\multicolumn{3}{c}{\bf Total Average:} 					&	928.16	&(	292	)&\bf	1680.95	&\bf(	276	)&	96.54	&\bf	41.27	&	61.20	&\bf	40.52	&	2451	&\bf	1819	&	711082	&\bf	237	&	5823	&\bf	82	\\

\hline
\end{tabular}%
\end{adjustbox}
   \caption{Results for \cite{EWC74} instances for $\ell_3$-norm \label{tab:EilonL3}}}
\end{table}%

\end{appendices}

\end{document}